\documentclass[12pt,twoside]{amsart}
\usepackage[latin1]{inputenc}
\usepackage{amsmath, amsthm, amscd, amsfonts, amssymb, graphicx}
\usepackage[bookmarksnumbered, plainpages]{hyperref}

\newtheorem{thm}{Theorem}[section]

\newtheorem{lem}[thm]{Lemma}

\newtheorem{defn}[thm]{Definition}

\numberwithin{equation}{section}

\begin{document}

\title{\bf The Kastler-Kalau-Walze type theorems about $J$-Witten deformation}
\author{Siyao Liu \hskip 0.4 true cm  Yong Wang$^{*}$}

\thanks{{\scriptsize
\hskip -0.4 true cm \textit{2010 Mathematics Subject Classification:}
53C40; 53C42.
\newline \textit{$^{*}$Corresponding author}}}

\maketitle

\begin{abstract}
 \indent In this paper, we obtain a Lichnerowicz type formula for $J$-Witten deformation and give the proof of the Kastler-Kalau-Walze type theorems associated with $J$-Witten deformation on four-dimensional and six-dimensional almost product Riemannian spin manifold with (respectively without) boundary.
\end{abstract}

\vskip 0.2 true cm


\pagestyle{myheadings}
\markboth{\rightline {\scriptsize Liu}}
         {\leftline{\scriptsize The Kastler-Kalau-Walze type theorems about $J$-Witten deformation}}

\bigskip
\bigskip


\section{ Introduction}
Based on the noncommutative residue found in \cite{Gu,Wo}, Connes claimed the noncommutative residue of the square of the inverse of the Dirac operator was proportioned to the Einstein-Hilbert action, which is called the Kastler-Kalau-Walze type theorem now.
This theorem was studied extensively by geometers \cite{Co1,Co2,Ka,KW,Ac,U}.
Wang generalized some results to the case of manifolds with boundary in \cite{Wa1, Wa2} and proved the Kastler-Kalau-Walze type theorems for the Dirac operator and the signature operator on lower-dimensional manifolds with boundary.

Most of the operators which have the leading symbol $\sqrt{-1}c(\xi),$ regarding the Kastler-Kalau-Walze theorem, have been studied extensively by, among others, the second author and by previous researchers \cite{Wa3,Wa4,Wa5,WW,WWW,Wa,LW1}.
Kim had given some preliminaries and lemmas about the Dirac operator $D$ and the $J$-twist in \cite{K}.
In \cite{Chen1,Chen2}, the author checked that $D_{J}$ is a formally self-adjoint elliptic operator.
By simple calculations, the leading symbol of the $J$-twist $D_{J}$ of the Dirac operator is not $\sqrt{-1}c(\xi)$.
In \cite{LW2,LW3}, Liu and Wang proved the Kastler-Kalau-Walze type theorems for the $J$-twist $D_{J}$ of the Dirac operator on almost product Riemannian spin manifold with boundary.
Zhang introduced the definition of an elliptic differential operator-Witten deformation in \cite{Z}.
Naturally, we can combine the $J$-twist $D_{J}$ of the Dirac operator and elliptic differential operator-Witten deformation and study the Kastler-Kalau-Walze theorem related to them.

For innovation, we choose the de Rham Hodge operator $\widetilde{D}$ to study in this paper.
The concepts of $J$-twist $\widetilde{D}_{J}$ of the de Rham Hodge operator and $J$-Witten deformation are defined.
The motivation of this paper is to prove the Kastler-Kalau-Walze type theorem associated with $J$-Witten deformation.

This paper is organized as follows.
In Section $2$, we first define the basic notions of $J$-twist $\widetilde{D}_{J}$ of the de Rham Hodge operator and $J$-Witten deformation.
We also give a Lichnerowicz type formula about $J$-Witten deformation and a Kastler-Kalau-Walze type theorem for $J$-Witten deformation on n-dimensional almost product Riemannian spin manifold without boundary in this section.
In the next section, we calculate $\widetilde{{\rm Wres}}[\pi^+\widetilde{D}_{W}^{-1}\circ\pi^+\widetilde{D}_{W}^{-1}]$ on four-dimensional almost product Riemannian spin manifold with boundary.
In Section $4$, we prove the Kastler-Kalau-Walze type theorem about $J$-Witten deformation on six-dimensional almost product Riemannian spin manifold with boundary.


\vskip 1 true cm

\section{ $J$-Witten deformation}

We give some definitions and basic notions that we will use in this paper.

Let $M$ be a $n$-dimensional ($n\geq 3$) oriented compact Riemannian manifold with a Riemannian metric $g^{M}$.
And let $\nabla^L$ be the Levi-Civita connection about $g^M$. In the local coordinates $\{x_i; 1\leq i\leq n\}$ and the fixed orthonormal frame $\{e_1,\cdots,e_n\}$, the connection matrix $(\omega_{s,t})$ is defined by
\begin{equation}
\nabla^L(e_1,\cdots,e_n)= (e_1,\cdots,e_n)(\omega_{s,t}).
\end{equation}
Let $\epsilon (e_j*)$,~$\iota (e_j*)$ be the exterior and interior multiplications respectively, $e_j*$ be the dual base of $e_j$  and $c(e_j)$ be the Clifford action.
Write
\begin{align}
c(e_j)=\epsilon (e_j*)-\iota (e_j*);\ \overline{c}(e_j)=\epsilon (e_j*)+\iota (e_j*),
\end{align}
which satisfies
\begin{align}
&c(e_i)c(e_j)+c(e_j)c(e_i)=-2\delta_i^j;\\
&\overline{c}(e_i)c(e_j)+c(e_j)\overline{c}(e_i)=0;\nonumber\\
&\overline{c}(e_i)\overline{c}(e_j)+\overline{c}(e_j)\overline{c}(e_i)=2\delta_i^j.\nonumber
\end{align}
Furthermore, we suppose that $\partial_{i}$ is a natural local frame on $TM$ and $(g^{ij})_{1\leq i,j\leq n}$ is the inverse matrix associated to the metric matrix  $(g_{ij})_{1\leq i,j\leq n}$ on $M$.
By \cite{WW2}, we have
\begin{align}
\widetilde{D}=\sum^n_{i=1}c(e_i)\bigg[e_i+\frac{1}{4}\sum_{s,t=1}^{n}\omega_{s,t}(e_i)[\overline{c}(e_s)\overline{c}(e_t)-c(e_s)c(e_t)]\bigg].
\end{align}
Let $g^{ij}=g(dx_{i},dx_{j})$, $\xi=\sum_{j}\xi_{j}dx_{j}$ and $\nabla^L_{\partial_{i}}\partial_{j}=\sum_{k}\Gamma_{ij}^{k}\partial_{k}$,  we denote that
\begin{align}
&\sigma_{i}=-\frac{1}{4}\sum_{s,t=1}^{n}\omega_{s,t}(e_i)c(e_s)c(e_t);\ a_{i}=\frac{1}{4}\sum_{s,t=1}^{n}\omega_{s,t}(e_i)\overline{c}(e_s)\overline{c}(e_t);\\
&\xi^{j}=g^{ij}\xi_{i};\ \partial^{j}=g^{ij}\partial_{i};\ \Gamma^{k}=g^{ij}\Gamma_{ij}^{k};\ \sigma^{j}=g^{ij}\sigma_{i};\ a^{j}=g^{ij}a_{i}.\nonumber
\end{align}
Then, the de Rham Hodge operator $\widetilde{D}$ can be written as
\begin{equation}
\widetilde{D}=\sum_{i, j=1}^{n}g^{ij}c(\partial_{i})\nabla_{\partial_{j}}^{\wedge^*T^*M}=\sum_{i=1}^{n}c(e_{i})\nabla_{e_{i}}^{\wedge^*T^*M},
\end{equation}
where
\begin{align}
&\nabla_{e_{i}}^{\wedge^*T^*M}=e_{i}+\sigma_{i}+a_{i}.
\end{align}

Let $J$ be a $(1, 1)$-tensor field on $(M, g^M)$ such that $J^2=\texttt{id},$
\begin{align}
&g^M(J(X), J(Y))=g^M(X, Y),
\end{align}
for all vector fields $X,Y\in \Gamma(TM).$ Here $\texttt{id}$ stands for the identity map. $(M, g^M, J)$ is an almost product Riemannian manifold.
We can define on almost product Riemannian spin manifold the following $J$-twist $\widetilde{D}_{J}$ of the de Rham Hodge operator $\widetilde{D}$ by
\begin{align}
&\widetilde{D}_{J}:=\sum_{i=1}^{n}c(e_{i})\nabla^{\wedge^*T^*M}_{J(e_{i})}=\sum_{i=1}^{n}c[J(e_{i})]\nabla^{\wedge^*T^*M}_{e_{i}}.
\end{align}
It is not difficult to check that $\widetilde{D}_{J}$ is an elliptic operator.

Based on the definition of the Witten deformation, we can define the $J$-Witten deformation as follows:
\begin{align}
&\widetilde{D}_{W}:=\sum_{i=1}^{n}c[J(e_{i})]\bigg[e_i+\frac{1}{4}\sum_{s,t=1}^{n}\omega_{s,t}(e_i)[\overline{c}(e_s)\overline{c}(e_t)-c(e_s)c(e_t)]\bigg]+\overline{c}(V),
\end{align}
where $V$ is a vector field.

By some simple calculations, we get the Lichnerowicz formula.
\begin{thm} The following equation holds:
\begin{align}
{\widetilde{D}_{W}}^{2}
&=-g^{ij}(\nabla_{\partial_{i}}\nabla_{\partial_{j}}-\nabla_{\nabla^{L}_{\partial_{i}}\partial_{j}})-\frac{1}{8}\sum_{i,j,k,l=1}^{n}R(J(e_{i}), J(e_{j}), e_{k}, e_{l})\overline{c}(e_{i})\overline{c}(e_{j})c(e_{k})c(e_{l})+\frac{1}{4}s\\
&-\frac{1}{2}\sum_{\nu,j=1}^{n}c[\nabla_{e_{j}}^{L}(J)e_{\nu}]c[(\nabla^{L}_{e_{\nu}}J)e_{j}]-\frac{1}{2}\sum_{\nu,j=1}^{n}c[J(e_{\nu})]c[(\nabla^{L}_{e_{j}}(\nabla^{L}_{e_{\nu}}(J)))e_{j}-(\nabla^{L}_{\nabla^{L}_{e_{j}}e_{\nu}}(J))e_{j}]\nonumber\\
&+\frac{1}{4}\sum_{\alpha,\nu,j=1}^{n}c[J(e_{\alpha})]c[(\nabla^{L}_{e_{\alpha}}J)e_{j}]c[J(e_{\nu})]c[(\nabla^{L}_{e_{\nu}}J)e_{j}]+\sum_{i=1}^{n}c[J(e_{i})]\overline{c}(\nabla^{L}_{e_{i}}V)+|V|^{2},\nonumber
\end{align}
where $s$ is the scalar curvature.
\end{thm}
\begin{proof}
Let $M$ be a smooth compact oriented Riemannian $n$-dimensional manifolds without boundary and $N$ be a vector bundle on $M$.
We say that $P$ is a differential operator of Laplace type, if it has locally the form
\begin{equation}
P=-(g^{ij}\partial_i\partial_j+A^i\partial_i+B),
\end{equation}
where $A^{i}$ and $B$ are smooth sections of $\textrm{End}(N)$ on $M$ (endomorphism).
If $P$ satisfies the form (2.12), then there is a unique connection $\nabla$ on $N$ and a unique endomorphism $E$ such that
\begin{equation}
P=-[g^{ij}(\nabla_{\partial_{i}}\nabla_{\partial_{j}}- \nabla_{\nabla^{L}_{\partial_{i}}\partial_{j}})+E].
\end{equation}
Moreover (with local frames of $T^{*}M$ and $N$), $\nabla_{\partial_{i}}=\partial_{i}+\omega_{i}$ and $E$ are related to $g^{ij}$, $A^{i}$ and $B$ through
\begin{eqnarray}
&&\omega_{i}=\frac{1}{2}g_{ij}\big(A^{i}+g^{kl}\Gamma_{kl}^{j} \texttt{id}\big),\\
&&E=B-g^{ij}\big(\partial_{i}(\omega_{j})+\omega_{i}\omega_{j}-\omega_{k}\Gamma_{ ij}^{k} \big),
\end{eqnarray}
where $\Gamma_{ kl}^{j}$ is the  Christoffel coefficient of $\nabla^{L}$.

We note that
\begin{align}
{\widetilde{D}_{W}}^{2}&={\widetilde{D}_{J}}^{2}+\widetilde{D}_{J}\overline{c}(V)+\overline{c}(V)\widetilde{D}_{J}+(\overline{c}(V))^2.
\end{align}
By (2.4) in \cite{K},
\begin{align}
{\widetilde{D}_{J}}^{2}&=\frac{1}{2}\sum_{i,j=1}^{n}c(e_{i})c(e_{j})R^{\wedge^*T^*M}(J(e_{i}), J(e_{j}))+\triangle+\sum_{\alpha,\beta=1}^{n}c[J(e_{\alpha})]c[(\nabla^{L}_{e_{\alpha}}J)e_{\beta}]\nabla^{\wedge^*T^*M}_{e_{\beta}}.
\end{align}
According to the formulas in \cite{Y}, we can get
\begin{align}
&\frac{1}{2}\sum_{i,j=1}^{n}c(e_{i})c(e_{j})R^{\wedge^*T^*M}(J(e_{i}), J(e_{j}))\\
&=-\frac{1}{8}\sum_{i,j,k,l=1}^{n}R(J(e_{i}), J(e_{j}), e_{k}, e_{l})\overline{c}(e_{i})\overline{c}(e_{j})c(e_{k})c(e_{l})+\frac{1}{4}s\nonumber
\end{align}
and
\begin{align}
\triangle&=-\sum_{i,j=1}^{n}g^{ij}[\partial_{i}\partial_{j}+2\sigma_{i}\partial_{j}+2a_{i}\partial_{j}-\Gamma_{ij}^{k}\partial_{k}
+(\partial_{i}\sigma_{j})+(\partial_{i}a_{j})+\sigma_{i}\sigma_{j}\\
&+\sigma_{i}a_{j}+a_{i}\sigma_{j}+a_{i}a_{j}-\Gamma_{ij}^{k}\sigma_{k}-\Gamma_{ij}^{k}a_{k}].\nonumber
\end{align}
We thus get
\begin{align}
{\widetilde{D}_{J}}^{2}&=-\frac{1}{8}\sum_{i,j,k,l=1}^{n}R(J(e_{i}), J(e_{j}), e_{k}, e_{l})\overline{c}(e_{i})\overline{c}(e_{j})c(e_{k})c(e_{l})+\frac{1}{4}s-\sum_{i,j=1}^{n}g^{ij}[\partial_{i}\partial_{j}\\
&+2\sigma_{i}\partial_{j}+2a_{i}\partial_{j}-\Gamma_{ij}^{k}\partial_{k}+(\partial_{i}\sigma_{j})+(\partial_{i}a_{j})+\sigma_{i}\sigma_{j}+\sigma_{i}a_{j}+a_{i}\sigma_{j}+a_{i}a_{j}\nonumber
\end{align}
\begin{align}
&-\Gamma_{ij}^{k}\sigma_{k}-\Gamma_{ij}^{k}a_{k}]+\sum_{\alpha,\beta=1}^{n}c[J(e_{\alpha})]c[(\nabla^{L}_{e_{\alpha}}J)e_{\beta}]\sum_{\gamma=1}^{n}\langle e_{\beta}, dx_{\gamma}\rangle\nabla^{\wedge^*T^*M}_{\partial_{\gamma}}\nonumber\\
&=-\frac{1}{8}\sum_{i,j,k,l=1}^{n}R(J(e_{i}), J(e_{j}), e_{k}, e_{l})\overline{c}(e_{i})\overline{c}(e_{j})c(e_{k})c(e_{l})+\frac{1}{4}s-\sum_{i,j=1}^{n}g^{ij}[\partial_{i}\partial_{j}\nonumber\\
&+2\sigma_{i}\partial_{j}+2a_{i}\partial_{j}-\Gamma_{ij}^{k}\partial_{k}+(\partial_{i}\sigma_{j})+(\partial_{i}a_{j})+\sigma_{i}\sigma_{j}+\sigma_{i}a_{j}+a_{i}\sigma_{j}+a_{i}a_{j}\nonumber\\
&-\Gamma_{ij}^{k}\sigma_{k}-\Gamma_{ij}^{k}a_{k}]+\sum_{\alpha,\gamma=1}^{n}c[J(e_{\alpha})]c[(\nabla^{L}_{e_{\alpha}}J)(dx_{\gamma})^{*}]\nabla^{\wedge^*T^*M}_{\partial_{\gamma}},\nonumber
\end{align}
where $e_i^*=g^{M}(e_i,\cdot)$ and $\langle X, dx_{\gamma}\rangle=g^{M}(X, (dx_{\gamma})^{*}),$ for a vector field $X.$
Computations show that
\begin{align}
\widetilde{D}_{J}\overline{c}(V)+\overline{c}(V)\widetilde{D}_{J}=\sum_{i=1}^{n}c[J(e_{i})]\overline{c}(\nabla^{L}_{e_{i}}V)
\end{align}
and
\begin{align}
(\overline{c}(V))^2=|V|^{2}.
\end{align}
Summarizing, we have
\begin{align}
{\widetilde{D}_{W}}^{2}
&=-\frac{1}{8}\sum_{i,j,k,l=1}^{n}R(J(e_{i}), J(e_{j}), e_{k}, e_{l})\overline{c}(e_{i})\overline{c}(e_{j})c(e_{k})c(e_{l})+\frac{1}{4}s-\sum_{i,j=1}^{n}g^{ij}[\partial_{i}\partial_{j}+2\sigma_{i}\partial_{j}\\
&+2a_{i}\partial_{j}-\Gamma_{ij}^{k}\partial_{k}+(\partial_{i}\sigma_{j})+(\partial_{i}a_{j})+\sigma_{i}\sigma_{j}+\sigma_{i}a_{j}+a_{i}\sigma_{j}+a_{i}a_{j}-\Gamma_{ij}^{k}\sigma_{k}-\Gamma_{ij}^{k}a_{k}]\nonumber\\
&+\sum_{\alpha,\gamma=1}^{n}c[J(e_{\alpha})]c[(\nabla^{L}_{e_{\alpha}}J)(dx_{\gamma})^{*}]\nabla^{\wedge^*T^*M}_{\partial_{\gamma}}+\sum_{i=1}^{n}c[J(e_{i})]\overline{c}(\nabla^{L}_{e_{i}}V)+|V|^{2}.\nonumber
\end{align}
Hence,
\begin{align}
(\omega_{i})_{{\widetilde{D}_{W}}^{2}}=\sigma_{i}+a_{i}-\frac{1}{2}\sum_{\alpha,p=1}^{n}g_{ip}c[J(e_{\alpha})]c[(\nabla^{L}_{e_{\alpha}}J)(dx_{p})^{*}],
\end{align}
\begin{align}
E_{{\widetilde{D}_{W}}^{2}}&=\frac{1}{8}\sum_{i,j,k,l=1}^{n}R(J(e_{i}), J(e_{j}), e_{k}, e_{l})\overline{c}(e_{i})\overline{c}(e_{j})c(e_{k})c(e_{l})-\frac{1}{4}s-|V|^{2}\\
&-\sum_{i=1}^{n}c[J(e_{i})]\overline{c}(\nabla^{L}_{e_{i}}V)+\sum_{i,j=1}^{n}g^{ij}[(\partial_{i}\sigma_{j})+(\partial_{i}a_{j})+\sigma_{i}\sigma_{j}+\sigma_{i}a_{j}+a_{i}\sigma_{j}\nonumber\\
&+a_{i}a_{j}-\Gamma_{ij}^{k}\sigma_{k}-\Gamma_{ij}^{k}a_{k}]-\sum_{\alpha,j=1}^{n}c[J(e_{\alpha})]c[(\nabla^{L}_{e_{\alpha}}J)(dx_{j})^{*}](\sigma_{j}+a_{j})\nonumber
\end{align}
\begin{align}
&-\sum_{i,j=1}^{n}g^{ij}\big[\partial_{i}(\sigma_{j}+a_{j}-\frac{1}{2}\sum_{\nu,q=1}^{n}g_{jq}c[J(e_{\nu})]c[(\nabla^{L}_{e_{\nu}}J)(dx_{q})^{*}])\nonumber\\
&+(\sigma_{i}+a_{i}-\frac{1}{2}\sum_{\alpha,p=1}^{n}g_{ip}c[J(e_{\alpha})]c[(\nabla^{L}_{e_{\alpha}}J)(dx_{p})^{*}])\nonumber\\
&\times(\sigma_{j}+a_{j}-\frac{1}{2}\sum_{\nu,l=1}^{n}g_{jl}c[J(e_{\nu})]c[(\nabla^{L}_{e_{\nu}}J)(dx_{l})^{*}])\nonumber\\
&-(\sigma_{k}+a_{k}-\frac{1}{2}\sum_{\mu,h=1}^{n}g_{kh}c[J(e_{\mu})]c[(\nabla^{L}_{e_{\mu}}J)(dx_{h})^{*}])\Gamma_{ij}^{k}\big].\nonumber
\end{align}
Since $E$ is globally defined on $M$, taking normal coordinates at $x_0$, we have $\sigma^{i}(x_0)=0,$ $a^{i}(x_0)=0,$ $\Gamma^k(x_0)=0,$ $g^{ij}(x_0)=\delta^j_i,$ $\partial^{j}(x_0)=e_{j},$ $\partial^{j}[c(\partial_{j})](x_0)=0,$ $\nabla^{L}_{e_{j}}e_{k}(x_0)=0$ and $\nabla^{\wedge^*T^*M}_{Y}(c(X))=c(\nabla^{L}_{Y}X),$ for vector fields X and Y, a simple calculation shows that
\begin{align}
E_{{\widetilde{D}_{W}}^{2}}(x_0)&=\frac{1}{8}\sum_{i,j,k,l=1}^{n}R(J(e_{i}), J(e_{j}), e_{k}, e_{l})\overline{c}(e_{i})\overline{c}(e_{j})c(e_{k})c(e_{l})-\frac{1}{4}s\\
&-\sum_{i=1}^{n}c[J(e_{i})]\overline{c}(\nabla^{L}_{e_{i}}V)+\frac{1}{2}\sum_{\nu,j=1}^{n}c[\nabla_{e_{j}}^{L}(J)e_{\nu}]c[(\nabla^{L}_{e_{\nu}}J)e_{j}]\nonumber\\
&+\frac{1}{2}\sum_{\nu,j=1}^{n}c[J(e_{\nu})]c[(\nabla^{L}_{e_{j}}(\nabla^{L}_{e_{\nu}}(J)))e_{j}-(\nabla^{L}_{\nabla^{L}_{e_{j}}e_{\nu}}(J))e_{j}]\nonumber\\
&-\frac{1}{4}\sum_{\alpha,\nu,j=1}^{n}c[J(e_{\alpha})]c[(\nabla^{L}_{e_{\alpha}}J)e_{j}]c[J(e_{\nu})]c[(\nabla^{L}_{e_{\nu}}J)e_{j}]-|V|^{2}.\nonumber
\end{align}
We should use (2.13) here, which completes the proof.
\end{proof}

According to the detailed descriptions in \cite{Ac}, we know that the noncommutative residue of a generalized laplacian $\widetilde{\Delta}$ is expressed as
\begin{equation}
(n-2)\Phi_{2}(\widetilde{\Delta})=(4\pi)^{\frac{n}{2}}\Gamma(\frac{n}{2})\widetilde{res}(\widetilde{\Delta}^{-\frac{n}{2}+1}),
\end{equation}
where $\Phi_{2}(\widetilde{\Delta})$ denotes the integral over the diagonal part of the second coefficient of the heat kernel expansion of $\widetilde{\Delta}$.
Now let $\widetilde{\Delta}={\widetilde{D}_{W}}^{2}$. Since ${\widetilde{D}_{W}}^{2}$ is a generalized laplacian, we can suppose ${\widetilde{D}_{W}}^{2}=\Delta-E$, then, we have
\begin{align}
{\rm Wres}({\widetilde{D}_{W}}^{2})^{-\frac{n-2}{2}}
=\frac{(n-2)\pi^{\frac{n}{2}}}{(\frac{n}{2}-1)!}\int_{M}{\rm tr}(-\frac{1}{6}s+E_{{\widetilde{D}_{W}}^{2}})d{\rm Vol_{M} },
\end{align}
where ${\rm Wres}$ denote the noncommutative residue, ${\rm tr}$ denote ${\rm trace}$.

\begin{thm}\cite{LW2} If $M$ is a $n$-dimensional almost product Riemannian spin manifold without boundary, we have the following:
\begin{align}
{\rm Wres}({\widetilde{D}_{W}}^{2})^{-\frac{n-2}{2}}
=\frac{(n-2)\pi^{\frac{n}{2}}}{(\frac{n}{2}-1)!}\int_{M}2^{n}\Big(&-\frac{5}{12}s-|V|^{2}-\frac{1}{2}\sum_{\nu,j=1}^{n}g^{M}(\nabla_{e_{j}}^{L}(J)e_{\nu}, (\nabla^{L}_{e_{\nu}}J)e_{j})\\
&-\frac{1}{2}\sum_{\nu,j=1}^{n}g^{M}(J(e_{\nu}), (\nabla^{L}_{e_{j}}(\nabla^{L}_{e_{\nu}}(J)))e_{j}-(\nabla^{L}_{\nabla^{L}_{e_{j}}e_{\nu}}(J))e_{j})\nonumber\\
&-\frac{1}{4}\sum_{\alpha,\nu,j=1}^{n}g^{M}(J(e_{\alpha}), (\nabla^{L}_{e_{\nu}}J)e_{j})g^{M}((\nabla^{L}_{e_{\alpha}}J)e_{j}, J(e_{\nu}))\nonumber\\
&-\frac{1}{4}\sum_{\alpha,\nu,j=1}^{n}g^{M}(J(e_{\alpha}), (\nabla^{L}_{e_{\alpha}}J)e_{j})g^{M}(J(e_{\nu}), (\nabla^{L}_{e_{\nu}}J)e_{j})\nonumber\\
&+\frac{1}{4}\sum_{\nu,j=1}^{n}g^{M}((\nabla^{L}_{e_{\nu}}J)e_{j}, (\nabla^{L}_{e_{\nu}}J)e_{j})\Big)d{\rm Vol_{M} },\nonumber
\end{align}
where $s$ is the scalar curvature.
\end{thm}
\begin{proof}
Set $X, Y, Z, W$ be the vector fields, we have
\begin{align}
{\rm tr}[c(X)\overline{c}(Y)]=0,
\end{align}
\begin{align}
{\rm tr}[c(X)c(Y)]=-g^{M}(X, Y){\rm tr}[\texttt{id}],
\end{align}
\begin{align}
{\rm tr}[\overline{c}(X)\overline{c}(Y)c(Z)c(W)]=-g^{M}(X, Y)g^{M}(Z, W){\rm tr}[\texttt{id}]
\end{align}
and
\begin{align}
{\rm tr}[c(X)c(Y)c(Z)c(W)]&=g^{M}(X, W)g^{M}(Y, Z){\rm tr}[\texttt{id}]-g^{M}(X, Z)g^{M}(Y, W){\rm tr}[\texttt{id}]\\
&+g^{M}(X, Y)g^{M}(Z, W){\rm tr}[\texttt{id}].\nonumber
\end{align}
Thus,
\begin{align}
\sum_{i,j,k,l=1}^{n}{\rm tr}[R(J(e_{i}), J(e_{j}), e_{k}, e_{l})\overline{c}(e_{i})\overline{c}(e_{j})c(e_{k})c(e_{l})]=0,
\end{align}
\begin{align}
&\sum_{\nu,j=1}^{n}{\rm tr}[c[\nabla_{e_{j}}^{L}(J)e_{\nu}]c[(\nabla^{L}_{e_{\nu}}J)e_{j}]]=-\sum_{\nu,j=1}^{n}g^{M}(\nabla_{e_{j}}^{L}(J)e_{\nu}, (\nabla^{L}_{e_{\nu}}J)e_{j}){\rm tr}[\texttt{id}],
\end{align}
\begin{align}
&\sum_{\nu,j=1}^{n}{\rm tr}[c[J(e_{\nu})]c[(\nabla^{L}_{e_{j}}(\nabla^{L}_{e_{\nu}}(J)))e_{j}-(\nabla^{L}_{\nabla^{L}_{e_{j}}e_{\nu}}(J))e_{j}]]
\end{align}
\begin{align}
&=-\sum_{\nu,j=1}^{n}g^{M}(J(e_{\nu}), (\nabla^{L}_{e_{j}}(\nabla^{L}_{e_{\nu}}(J)))e_{j}-(\nabla^{L}_{\nabla^{L}_{e_{j}}e_{\nu}}(J))e_{j}){\rm tr}[\texttt{id}],\nonumber
\end{align}
\begin{align}
&\sum_{\alpha,\nu,j=1}^{n}{\rm tr}[c[J(e_{\alpha})]c[(\nabla^{L}_{e_{\alpha}}J)e_{j}]c[J(e_{\nu})]c[(\nabla^{L}_{e_{\nu}}J)e_{j}]]\\
&=\sum_{\alpha,\nu,j=1}^{n}g^{M}(J(e_{\alpha}), (\nabla^{L}_{e_{\nu}}J)e_{j})g^{M}((\nabla^{L}_{e_{\alpha}}J)e_{j}, J(e_{\nu})){\rm tr}[\texttt{id}]\nonumber\\
&-\sum_{\nu,j=1}^{n}g^{M}((\nabla^{L}_{e_{\nu}}J)e_{j}, (\nabla^{L}_{e_{\nu}}J)e_{j}){\rm tr}[\texttt{id}]\nonumber\\
&+\sum_{\alpha,\nu,j=1}^{n}g^{M}(J(e_{\alpha}), (\nabla^{L}_{e_{\alpha}}J)e_{j})g^{M}(J(e_{\nu}), (\nabla^{L}_{e_{\nu}}J)e_{j}){\rm tr}[\texttt{id}].\nonumber
\end{align}
By applying the formulas shown in (2.26) and (2.28), we obtain Theorem 2.2.
\end{proof}

\section{ The Kastler-Kalau-Walze type theorem for $4$-dimensional manifolds with boundary }

Firstly, we explain the basic notions of Boutet de Monvel's calculus and the definition of the noncommutative residue for manifolds with boundary that will be used throughout the paper.
For the details, see Ref.\cite{Wa3}.

Let $U\subset M$ be a collar neighborhood of $\partial M$ which is diffeomorphic with $\partial M\times [0,1)$.
By the definition of $h(x_n)\in C^{\infty}([0,1))$ and $h(x_n)>0$, there exists $\widehat{h}\in C^{\infty}((-\varepsilon,1))$ such that $\widehat{h}|_{[0,1)}=h$ and $\widehat{h}>0$ for some sufficiently small $\varepsilon>0$.
Then there exists a metric $g'$ on $\widetilde{M}=M\bigcup_{\partial M}\partial M\times(-\varepsilon,0]$ which has the form on $U\bigcup_{\partial M}\partial M\times (-\varepsilon,0 ]$
\begin{equation}
g'=\frac{1}{\widehat{h}(x_{n})}g^{\partial M}+dx _{n}^{2} ,
\end{equation}
such that $g'|_{M}=g$.
We fix a metric $g'$ on the $\widetilde{M}$ such that $g'|_{M}=g$.

We define the Fourier transformation $F'$  by
\begin{equation}
F':L^2({\bf R}_t)\rightarrow L^2({\bf R}_v);~F'(u)(v)=\int e^{-ivt}u(t)dt\\
\end{equation}
and let
\begin{equation}
r^{+}:C^\infty ({\bf R})\rightarrow C^\infty (\widetilde{{\bf R}^+});~ f\rightarrow f|\widetilde{{\bf R}^+};~
\widetilde{{\bf R}^+}=\{x\geq0;x\in {\bf R}\},
\end{equation}
where $\Phi({\bf R})$ denotes the Schwartz space and $\Phi(\widetilde{{\bf R}^+}) =r^+\Phi({\bf R})$, $\Phi(\widetilde{{\bf R}^-}) =r^-\Phi({\bf R})$.

We define $H^+=F'(\Phi(\widetilde{{\bf R}^+}));~ H^-_0=F'(\Phi(\widetilde{{\bf R}^-}))$ which satisfies $H^+\bot H^-_0$.
We have the following property: $h\in H^+~$ (resp. $H^-_0$) if and only if $h\in C^\infty({\bf R})$ which has an analytic extension to the lower (resp. upper) complex half-plane $\{{\rm Im}\xi<0\}$ (resp. $\{{\rm Im}\xi>0\})$ such that for all nonnegative integer $l$,
\begin{equation}
\frac{d^{l}h}{d\xi^l}(\xi)\sim\sum^{\infty}_{k=1}\frac{d^l}{d\xi^l}(\frac{c_k}{\xi^k}),
\end{equation}
as $|\xi|\rightarrow +\infty,$ ${\rm Im}\xi\leq0$ (resp. ${\rm Im}\xi\geq0$).

Let $H'$ be the space of all polynomials and $H^-=H^-_0\bigoplus H';~H=H^+\bigoplus H^-.$
Denote by $\pi^+$ (resp. $\pi^-$) the projection on $H^+$ (resp. $H^-$).
For calculations, we take $H=\widetilde H=\{$rational functions having no poles on the real axis$\}$ ($\tilde{H}$ is a dense set in the topology of $H$).
Then on $\tilde{H}$,
\begin{equation}
\pi^+h(\xi_0)=\frac{1}{2\pi i}\lim_{u\rightarrow 0^{-}}\int_{\Gamma^+}\frac{h(\xi)}{\xi_0+iu-\xi}d\xi,
\end{equation}
where $\Gamma^+$ is a Jordan closed curve included ${\rm Im}(\xi)>0$ surrounding all the singularities of $h$ in the upper half-plane and
$\xi_0\in {\bf R}$.
Similarly, define $\pi'$ on $\tilde{H}$,
\begin{equation}
\pi'h=\frac{1}{2\pi}\int_{\Gamma^+}h(\xi)d\xi.
\end{equation}
So, $\pi'(H^-)=0$.
For $h\in H\bigcap L^1({\bf R})$, $\pi'h=\frac{1}{2\pi}\int_{{\bf R}}h(v)dv$ and for $h\in H^+\bigcap L^1({\bf R})$, $\pi'h=0$.

Let $M$ be a $n$-dimensional compact oriented manifold with boundary $\partial M$.
Denote by $\mathcal{B}$ Boutet de Monvel's algebra, we recall the main theorem in \cite{Wa3,FGLS}.
\begin{thm}
\label{th:32}{\rm\cite{FGLS}}{\bf(Fedosov-Golse-Leichtnam-Schrohe)}
Let $X$ and $\partial X$ be connected, ${\rm dim}X=n\geq3$, $A=\left(
\begin{array}{lcr}
\pi^+P+G &   K \\T &  S
\end{array}\right)$ $\in \mathcal{B},$ and denote by $p$, $b$ and $s$ the local symbols of $P,G$ and $S$ respectively.
Define:
\begin{align}
{\rm{\widetilde{Wres}}}(A)&=\int_X\int_{\bf S}{\rm{tr}}_E\left[p_{-n}(x,\xi)\right]\sigma(\xi)dx \\
&+2\pi\int_ {\partial X}\int_{\bf S'}\left\{{\rm tr}_E\left[({\rm{tr}}b_{-n})(x',\xi')\right]+{\rm{tr}}
_F\left[s_{1-n}(x',\xi')\right]\right\}\sigma(\xi')dx',\nonumber
\end{align}
where ${\rm{\widetilde{Wres}}}$ denotes the noncommutative residue of an operator in the Boutet de Monvel's algebra.\\
Then~~ a) ${\rm \widetilde{Wres}}([A,B])=0 $, for any $A,B\in\mathcal{B}$;~~ b) It is a unique continuous trace on $\mathcal{B}/\mathcal{B}^{-\infty}$.
\end{thm}

\begin{defn}{\rm\cite{Wa3}}
Lower dimensional volumes of spin manifolds with boundary are defined by
 \begin{equation}
{\rm Vol}^{(p_1,p_2)}_nM:= \widetilde{{\rm Wres}}[\pi^+D^{-p_1}\circ\pi^+D^{-p_2}].
\end{equation}
\end{defn}

By \cite{Wa3}, we get
\begin{align}
\widetilde{{\rm Wres}}[\pi^+D^{-p_1}\circ\pi^+D^{-p_2}]=\int_M\int_{|\xi|=1}{\rm
tr}_{\wedge^*T^*M}[\sigma_{-n}(D^{-p_1-p_2})]\sigma(\xi)dx+\int_{\partial M}\Phi
\end{align}
and
\begin{align}
\Phi=&\int_{|\xi'|=1}\int^{+\infty}_{-\infty}\sum^{\infty}_{j, k=0}\sum\frac{(-i)^{|\alpha|+j+k+1}}{\alpha!(j+k+1)!}
\times {\rm tr}_{\wedge^*T^*M}[\partial^j_{x_n}\partial^\alpha_{\xi'}\partial^k_{\xi_n}\sigma^+_{r}(D^{-p_1})(x',0,\xi',\xi_n)
\\
&\times\partial^\alpha_{x'}\partial^{j+1}_{\xi_n}\partial^k_{x_n}\sigma_{l}(D^{-p_2})(x',0,\xi',\xi_n)]d\xi_n\sigma(\xi')dx',\nonumber
\end{align}
 where the sum is taken over $r+l-k-|\alpha|-j-1=-n,~~r\leq -p_1,l\leq -p_2$.

Since $[\sigma_{-n}(D^{-p_1-p_2})]|_M$ has the same expression as $\sigma_{-n}(D^{-p_1-p_2})$ in the case of manifolds without
boundary, so locally we can compute the first term by \cite{Ka}, \cite{KW}, \cite{Wa3}, \cite{Po}.

For any fixed point $x_0\in\partial M$, we choose the normal coordinates $U$ of $x_0$ in $\partial M$ (not in $M$) and compute $\Phi(x_0)$ in the coordinates $\widetilde{U}=U\times [0,1)\subset M$ and the metric $\frac{1}{h(x_n)}g^{\partial M}+dx_n^2.$ The dual metric of $g^M$ on $\widetilde{U}$ is ${h(x_n)}g^{\partial M}+dx_n^2.$  Write $g^M_{ij}=g^M(\frac{\partial}{\partial x_i},\frac{\partial}{\partial x_j});~ g_M^{ij}=g^M(dx_i,dx_j)$, then
\begin{equation}
[g^M_{ij}]= \left[\begin{array}{lcr}
  \frac{1}{h(x_n)}[g_{ij}^{\partial M}]  & 0  \\
   0  &  1
\end{array}\right];~~~
[g_M^{ij}]= \left[\begin{array}{lcr}
  h(x_n)[g^{ij}_{\partial M}]  & 0  \\
   0  &  1
\end{array}\right]
\end{equation}
and
\begin{equation}
\partial_{x_s}g_{ij}^{\partial M}(x_0)=0, 1\leq i,j\leq n-1; ~~~g_{ij}^M(x_0)=\delta_{ij}.
\end{equation}

$\{e_1, \cdots, e_n\}$ be an orthonormal frame field in $U$ about $g^{\partial M}$ which is parallel along geodesics and $e_i(x_0)=\frac{\partial}{\partial{x_i}}(x_0).$
We review the following three lemmas.
\begin{lem}{\rm \cite{Wa3}}\label{le:32}
With the metric $g^{M}$ on $M$ near the boundary
\begin{eqnarray}
\partial_{x_j}(|\xi|_{g^M}^2)(x_0)&=&\left\{
       \begin{array}{c}
        0,  ~~~~~~~~~~ ~~~~~~~~~~ ~~~~~~~~~~~~~{\rm if }~j<n; \\[2pt]
       h'(0)|\xi'|^{2}_{g^{\partial M}},~~~~~~~~~~~~~~~~~~~~{\rm if }~j=n,
       \end{array}
    \right. \\
\partial_{x_j}[c(\xi)](x_0)&=&\left\{
       \begin{array}{c}
      0,  ~~~~~~~~~~ ~~~~~~~~~~ ~~~~~~~~~~~~~{\rm if }~j<n;\\[2pt]
\partial x_{n}(c(\xi'))(x_{0}), ~~~~~~~~~~~~~~~~~{\rm if }~j=n,
       \end{array}
    \right.
\end{eqnarray}
where $\xi=\xi'+\xi_{n}dx_{n}$.
\end{lem}
\begin{lem}{\rm \cite{Wa3}}\label{le:32}
With the metric $g^{M}$ on $M$ near the boundary
\begin{align}
\omega_{s,t}(e_i)(x_0)&=\left\{
       \begin{array}{c}
        \omega_{n,i}(e_i)(x_0)=\frac{1}{2}h'(0),  ~~~~~~~~~~ ~~~~~~~~~~~{\rm if }~s=n,t=i,i<n; \\[2pt]
       \omega_{i,n}(e_i)(x_0)=-\frac{1}{2}h'(0),~~~~~~~~~~~~~~~~~~~{\rm if }~s=i,t=n,i<n;\\[2pt]
    \omega_{s,t}(e_i)(x_0)=0,~~~~~~~~~~~~~~~~~~~~~~~~~~~other~cases,~~~~~~~~~\\[2pt]
       \end{array}
    \right.
\end{align}
where $(\omega_{s,t})$ denotes the connection matrix of Levi-Civita connection $\nabla^L$.
\end{lem}
\begin{lem}{\rm \cite{Wa3}}
When $i<n,$ then
\begin{align}
\label{b13}
\Gamma_{st}^k(x_0)&=\left\{
       \begin{array}{c}
        \Gamma^n_{ii}(x_0)=\frac{1}{2}h'(0),~~~~~~~~~~ ~~~~~~~~~~~{\rm if }~s=t=i,k=n; \\[2pt]
        \Gamma^i_{ni}(x_0)=-\frac{1}{2}h'(0),~~~~~~~~~~~~~~~~~~~{\rm if }~s=n,t=i,k=i;\\[2pt]
        \Gamma^i_{in}(x_0)=-\frac{1}{2}h'(0),~~~~~~~~~~~~~~~~~~~{\rm if }~s=i,t=n,k=i,\\[2pt]
       \end{array}
    \right.
\end{align}
in other cases, $\Gamma_{st}^i(x_0)=0$.
\end{lem}

Similar to (3.9) and (3.10), we firstly compute
\begin{equation}
\widetilde{{\rm Wres}}[\pi^+{\widetilde{D}_{W}}^{-1}\circ\pi^+\widetilde{D}_{W}^{-1}]=\int_M\int_{|\xi|=1}{\rm tr}_{\wedge^*T^*M}[\sigma_{-4}(\widetilde{D}_{W}^{-2})]\sigma(\xi)dx+\int_{\partial M}\Psi,
\end{equation}
where
\begin{align}
\Psi&=\int_{|\xi'|=1}\int^{+\infty}_{-\infty}\sum^{\infty}_{j, k=0}\sum\frac{(-i)^{|\alpha|+j+k+1}}{\alpha!(j+k+1)!}
\times {\rm tr}_{\wedge^*T^*M}[\partial^j_{x_n}\partial^\alpha_{\xi'}\partial^k_{\xi_n}\sigma^+_{r}(\widetilde{D}_{W}^{-1})\\
&(x',0,\xi',\xi_n)\times\partial^\alpha_{x'}\partial^{j+1}_{\xi_n}\partial^k_{x_n}\sigma_{l}(\widetilde{D}_{W}^{-1})(x',0,\xi',\xi_n)]d\xi_n\sigma(\xi')dx',\nonumber
\end{align}
the sum is taken over $r+l-k-j-|\alpha|-1=-4, r\leq -1, l\leq-1$.

Computations show that
\begin{align}
&\int_M\int_{|\xi|=1}{\rm tr}_{\wedge^*T^*M}[\sigma_{-4}({\widetilde{D}_{W}}^{-2})]\sigma(\xi)dx=8\pi^{2}\\
&\int_{M}\Big(-\frac{5}{3}s-4|V|^2
-2\sum_{\nu,j=1}^{4}g^{M}(\nabla_{e_{j}}^{L}(J)e_{\nu}, (\nabla^{L}_{e_{\nu}}J)e_{j})\nonumber\\
&-2\sum_{\nu,j=1}^{4}g^{M}(J(e_{\nu}), (\nabla^{L}_{e_{j}}(\nabla^{L}_{e_{\nu}}(J)))e_{j}-(\nabla^{L}_{\nabla^{L}_{e_{j}}e_{\nu}}(J))e_{j})\nonumber\\
&-\sum_{\alpha,\nu,j=1}^{4}g^{M}(J(e_{\alpha}), (\nabla^{L}_{e_{\nu}}J)e_{j})g^{M}((\nabla^{L}_{e_{\alpha}}J)e_{j}, J(e_{\nu}))\nonumber\\
&-\sum_{\alpha,\nu,j=1}^{4}g^{M}(J(e_{\alpha}), (\nabla^{L}_{e_{\alpha}}J)e_{j})g^{M}(J(e_{\nu}), (\nabla^{L}_{e_{\nu}}J)e_{j})\nonumber\\
&+\sum_{\nu,j=1}^{4}g^{M}((\nabla^{L}_{e_{\nu}}J)e_{j}, (\nabla^{L}_{e_{\nu}}J)e_{j})\Big)d{\rm Vol_{M} }.\nonumber
\end{align}

Now, we compute $\int_{\partial M}\Psi.$
The operator have the following symbols.
\begin{lem} The following identities hold:
\begin{align}
\sigma_1(\widetilde{D}_{W})&=ic[J(\xi)];\\
\sigma_0(\widetilde{D}_{W})&=
\frac{1}{4}\sum_{i,s,t=1}^{n}\omega_{s,t}(e_i)c[J(e_i)]\overline{c}(e_s)\overline{c}(e_t)-\frac{1}{4}\sum_{i,s,t=1}^{n}\omega_{s,t}(e_i)c[J(e_i)]c(e_s)c(e_t)+\overline{c}(V).
\end{align}
\end{lem}

Write
\begin{eqnarray}
D_x^{\alpha}&=(-i)^{|\alpha|}\partial_x^{\alpha};
~\sigma(\widetilde{D}_{W})=p_1+p_0;
~\sigma({\widetilde{D}_{W}}^{-1})=\sum^{\infty}_{j=1}q_{-j}.
\end{eqnarray}
By the composition formula of pseudodifferential operators, we have
\begin{align}
1=\sigma(\widetilde{D}_{W}\circ {\widetilde{D}_{W}}^{-1})&=\sum_{\alpha}\frac{1}{\alpha!}\partial^{\alpha}_{\xi}[\sigma(\widetilde{D}_{W})]
{D}_x^{\alpha}[\sigma({\widetilde{D}_{W}}^{-1})]\\
&=(p_1+p_0)(q_{-1}+q_{-2}+q_{-3}+\cdots)\nonumber\\
&+\sum_j(\partial_{\xi_j}p_1+\partial_{\xi_j}p_0)(
D_{x_j}q_{-1}+D_{x_j}q_{-2}+D_{x_j}q_{-3}+\cdots)\nonumber\\
&=p_1q_{-1}+(p_1q_{-2}+p_0q_{-1}+\sum_j\partial_{\xi_j}p_1D_{x_j}q_{-1})+\cdots,\nonumber
\end{align}
so
\begin{equation}
q_{-1}=p_1^{-1};~q_{-2}=-p_1^{-1}[p_0p_1^{-1}+\sum_j\partial_{\xi_j}p_1D_{x_j}(p_1^{-1})].
\end{equation}

\begin{lem} The following identities hold:
\begin{align}
\sigma_{-1}({\widetilde{D}_{W}}^{-1})&=\frac{ic[J(\xi)]}{|\xi|^2};\\
\sigma_{-2}({\widetilde{D}_{W}}^{-1})&=\frac{c[J(\xi)]\sigma_{0}(\widetilde{D}_{W})c[J(\xi)]}{|\xi|^4}+\frac{c[J(\xi)]}{|\xi|^6}\sum_ {j=1}^{n} c[J(dx_j)]
\Big[\partial_{x_j}(c[J(\xi)])|\xi|^2-c[J(\xi)]\partial_{x_j}(|\xi|^2)\Big].
\end{align}
\end{lem}

When $n=4$, then ${\rm tr}_{\wedge^*T^*M}[{\rm \texttt{id}}]={\rm dim}(\wedge^*(\mathbb{R}^4))=16,$ since the sum is taken over $r+l-k-j-|\alpha|-1=-4,~~r\leq -1,l\leq-1,$ then we have the following five cases:\\

{\bf case (a)~(I)}~$r=-1,~l=-1,~k=j=0,~|\alpha|=1$\\

By applying the formula shown in (3.18), we can calculate
\begin{align}
\Psi_1=-\int_{|\xi'|=1}\int^{+\infty}_{-\infty}\sum_{|\alpha|=1}
{\rm tr}_{\wedge^*T^*M}[\partial^\alpha_{\xi'}\pi^+_{\xi_n}\sigma_{-1}({\widetilde{D}_{W}}^{-1})\times \partial^\alpha_{x'}\partial_{\xi_n}\sigma_{-1}({\widetilde{D}_{W}}^{-1})](x_0)d\xi_n\sigma(\xi')dx'.
\end{align}

{\bf case (a)~(II)}~$r=-1,~l=-1,~k=|\alpha|=0,~j=1$\\

It is easy to check that
\begin{align}
\Psi_2=-\frac{1}{2}\int_{|\xi'|=1}\int^{+\infty}_{-\infty} {\rm tr}_{\wedge^*T^*M} [\partial_{x_n}\pi^+_{\xi_n}\sigma_{-1}({\widetilde{D}_{W}}^{-1})\times
\partial_{\xi_n}^2\sigma_{-1}({\widetilde{D}_{W}}^{-1})](x_0)d\xi_n\sigma(\xi')dx'.
\end{align}

{\bf case (a)~(III)}~$r=-1,~l=-1,~j=|\alpha|=0,~k=1$\\

By (3.18), we calculate that
\begin{align}
\Psi_3=-\frac{1}{2}\int_{|\xi'|=1}\int^{+\infty}_{-\infty}
{\rm tr}_{\wedge^*T^*M} [\partial_{\xi_n}\pi^+_{\xi_n}\sigma_{-1}({\widetilde{D}_{W}}^{-1})\times
\partial_{\xi_n}\partial_{x_n}\sigma_{-1}({\widetilde{D}_{W}}^{-1})](x_0)d\xi_n\sigma(\xi')dx'.
\end{align}

Similar to the formulas (3.28)-(3.36) in \cite{LW2}, we have
\begin{align}
\Psi_1+\Psi_2+\Psi_3=\sum_{\beta=1}^{n}\sum_{i=1}^{n-1}a_{\beta}^{i}\partial_{x_i}(a_{\beta}^{n}){\rm tr}[\texttt{id}]\Omega_3(-\frac{\pi}{8}+\frac{\pi^2}{3})dx'+\sum_{\beta=1}^{n}\sum_{i=1}^{n-1}a_{\beta}^{n}\partial_{x_i}(a_{\beta}^{i}){\rm tr}[\texttt{id}]\Omega_3(-\frac{\pi^2}{6})dx',
\end{align}
where $\Omega_{3}=\frac{2\pi^\frac{3}{2}}{\Gamma(\frac{3}{2})}.$\\

{\bf case (b)}~$r=-2,~l=-1,~k=j=|\alpha|=0$\\

Similarly, we get
\begin{align}
\Psi_4&=-i\int_{|\xi'|=1}\int^{+\infty}_{-\infty}{\rm tr}_{\wedge^*T^*M}[\pi^+_{\xi_n}\sigma_{-2}({\widetilde{D}_{W}}^{-1})\times
\partial_{\xi_n}\sigma_{-1}({\widetilde{D}_{W}}^{-1})](x_0)d\xi_n\sigma(\xi')dx'.
\end{align}

Let us first compute $\partial_{\xi_n}\sigma_{-1}({\widetilde{D}_{W}}^{-1})(x_0).$
\begin{align}
\partial_{\xi_n}\left(\frac{ic[J(\xi)]}{|\xi|^2}\right)(x_0)|_{|\xi'|=1}
&=i\sum^{n}_{\beta=1}\sum^{n-1}_{i=1}\xi_{i}a_{\beta}^ic(dx_{\beta})\partial_{\xi_n}\left(\frac{1}{1+\xi_{n}^2}\right)+i\sum^{n}_{\beta=1}a_{\beta}^nc(dx_{\beta})\partial_{\xi_n}\left(\frac{\xi_{n}}{1+\xi_{n}^2}\right)\\
&=-\frac{2i\xi_n}{(1+\xi_{n}^2)^2}\sum^{n}_{\beta=1}\sum^{n-1}_{i=1}\xi_{i}a_{\beta}^ic(dx_{\beta})+\frac{i(1-\xi_n^2)}{(1+\xi_{n}^2)^2}\sum^{n}_{\beta=1}a_{\beta}^nc(dx_{\beta}),\nonumber
\end{align}
where $J(dx_{p})=\sum^{n}_{h=1}a^{p}_{h}dx_{h}.$

We next calculate that
\begin{align}
\sigma_{-2}({\widetilde{D}_{W}}^{-1})(x_0)&=\frac{c[J(\xi)]\sigma_{0}(\widetilde{D}_{W})(x_0)c[J(\xi)]}{|\xi|^4}+\frac{c[J(\xi)]}{|\xi|^6}\sum_{j=1}^{n} c[J(dx_j)]\Big[\sum_{p,h=1}^{n}\xi_p\partial_{x_j}(a_{h}^{p})c(dx_h)|\xi|^2\\
&+\sum_{p,h=1}^{n}\xi_pa_{h}^{p}\partial_{x_j}(c(dx_h))|\xi|^2-c[J(\xi)]\partial_{x_j}(|\xi|^2)\Big](x_0)\nonumber\\
&=\frac{c[J(\xi)]\sigma_{0}(\widetilde{D}_{W})(x_0)c[J(\xi)]}{|\xi|^4}-\frac{c[J(\xi)]}{|\xi|^6}h'(0)|\xi'|^2c[J(dx_n)]c[J(\xi)]\nonumber\\
&+\frac{c[J(\xi)]}{|\xi|^4}\Big[\sum_{j,p,h=1}^{n}\xi_p\partial_{x_j}(a_{h}^{p})c[J(dx_j)]c(dx_h)+\sum_{p=1}^{n}\sum_{h=1}^{n-1}\xi_pa_{h}^{p}c[J(dx_n)]\partial_{x_n}(c(dx_h))\Big](x_0),\nonumber
\end{align}
where
\begin{align}
\sigma_{0}(\widetilde{D}_{W})(x_0)
&=\frac{1}{4}\sum_{i,s,t=1}^{n}\omega_{s,t}(e_i)(x_{0})c[J(e_i)]\overline{c}(e_s)\overline{c}(e_t)\\
&-\frac{1}{4}\sum_{i,s,t=1}^{n}\omega_{s,t}(e_i)(x_{0})c[J(e_i)]c(e_s)c(e_t)+\overline{c}(V)\nonumber\\
&=\frac{1}{4}h'(0)\sum_{\eta=1}^{n}\sum_{\nu=1}^{n-1}a_{\nu}^{\eta}c(dx_{\eta})\overline{c}(dx_{n})\overline{c}(dx_{\nu})\nonumber\\
&-\frac{1}{4}h'(0)\sum_{\mu=1}^{n}\sum_{\nu=1}^{n-1}a_{\nu}^{\mu}c(dx_{\mu})c(dx_{n})c(dx_{\nu})+\overline{c}(V).\nonumber
\end{align}
To shorten notation, we let
\begin{align}
\widetilde{b}_0^1(x_0)&=\frac{1}{4}h'(0)\sum_{\eta=1}^{n}\sum_{\nu=1}^{n-1}a_{\nu}^{\eta}c(dx_{\eta})\overline{c}(dx_{n})\overline{c}(dx_{\nu});\\
\widetilde{b}_0^2(x_0)&=-\frac{1}{4}h'(0)\sum_{\mu=1}^{n}\sum_{\nu=1}^{n-1}a_{\nu}^{\mu}c(dx_{\mu})c(dx_{n})c(dx_{\nu}),
\end{align}
means that
\begin{align}
&\pi^+_{\xi_n}\sigma_{-2}({\widetilde{D}_{W}}^{-1})(x_0)|_{|\xi'|=1}=\pi^+_{\xi_n}\Big(\frac{c[J(\xi)]\widetilde{b}_0^2(x_0)c[J(\xi)]}{(1+\xi_n^2)^2}\Big)-h'(0)\pi^+_{\xi_n}\Big(\frac{c[J(\xi)]}{(1+\xi_n^2)^3}c[J(dx_n)]c[J(\xi)]\Big)\\
&+\pi^+_{\xi_n}\Big(\frac{c[J(\xi)]}{(1+\xi_n^2)^2}\Big[\sum_{j,p,h=1}^{n}\xi_p\partial_{x_j}(a_{h}^{p})c[J(dx_j)]c(dx_h)+\sum_{p=1}^{n}\sum_{h=1}^{n-1}\xi_pa_{h}^{p}c[J(dx_n)]\partial_{x_n}(c(dx_h))\Big]\Big)\nonumber\\
&+\pi^+_{\xi_n}\Big(\frac{c[J(\xi)]\widetilde{b}_0^1(x_0)c[J(\xi)]}{(1+\xi_n^2)^2}\Big)+\pi^+_{\xi_n}\Big(\frac{c[J(\xi)]\overline{c}(V)c[J(\xi)]}{(1+\xi_n^2)^2}\Big).\nonumber
\end{align}

By
\begin{align}
\pi^+_{\xi_n}\left(\frac{1}{(1+\xi_{n}^2)^2}\right)(x_0)
&=\frac{1}{2\pi i}{\rm lim}_{u\rightarrow
0^-}\int_{\Gamma^+}\frac{\frac{1}{(\eta_n+i)^2(\xi_n+iu-\eta_n)}}{(\eta_n-i)^2}d\eta_n\\
&=\left[\frac{1}{(\eta_n+i)^2(\xi_n-\eta_n)}\right]^{(1)}|_{\eta_n=i}=-\frac{i\xi_n+2}{4(\xi_n-i)^2},\nonumber
\end{align}
\begin{align}
\pi^+_{\xi_n}\left(\frac{\xi_n}{(1+\xi_{n}^2)^2}\right)(x_0)
&=\frac{1}{2\pi i}{\rm lim}_{u\rightarrow
0^-}\int_{\Gamma^+}\frac{\frac{\eta_n}{(\eta_n+i)^2(\xi_n+iu-\eta_n)}}{(\eta_n-i)^2}d\eta_n\\
&=\left[\frac{\eta_n}{(\eta_n+i)^2(\xi_n-\eta_n)}\right]^{(1)}|_{\eta_n=i}=-\frac{i}{4(\xi_n-i)^2}\nonumber
\end{align}
and
\begin{align}
\pi^+_{\xi_n}\left(\frac{\xi_n^2}{(1+\xi_{n}^2)^2}\right)(x_0)
&=\frac{1}{2\pi i}{\rm lim}_{u\rightarrow
0^-}\int_{\Gamma^+}\frac{\frac{\eta_n^2}{(\eta_n+i)^2(\xi_n+iu-\eta_n)}}{(\eta_n-i)^2}d\eta_n\\
&=\left[\frac{\eta_n^2}{(\eta_n+i)^2(\xi_n-\eta_n)}\right]^{(1)}|_{\eta_n=i}=-\frac{i\xi_{n}}{4(\xi_n-i)^2},\nonumber
\end{align}
it is evident that
\begin{align}
\pi^+_{\xi_n}\Big(\frac{c[J(\xi)]\widetilde{b}_0^1(x_0)c[J(\xi)]}{(1+\xi_n^2)^2}\Big)&=-\frac{i\xi_n}{4(\xi_n-i)^2}\sum_{l,\gamma=1}^{n}a_{l}^{n}a_{\gamma}^{n}c(dx_l)\widetilde{b}_0^1(x_0)c(dx_{\gamma})\nonumber\\
&-\frac{i}{4(\xi_n-i)^2}\sum_{l,\gamma=1}^{n}\sum_{q=1}^{n-1}\xi_{q}a_{l}^{q}a_{\gamma}^{n}c(dx_l)\widetilde{b}_0^1(x_0)c(dx_{\gamma})\nonumber\\
&-\frac{i}{4(\xi_n-i)^2}\sum_{l,\gamma=1}^{n}\sum_{\alpha=1}^{n-1}\xi_{\alpha}a_{l}^{n}a_{\gamma}^{\alpha}c(dx_l)\widetilde{b}_0^1(x_0)c(dx_{\gamma})\nonumber\\
&-\frac{i\xi_n+2}{4(\xi_n-i)^2}\sum_{l,\gamma=1}^{n}\sum_{q,\alpha=1}^{n-1}\xi_{q}\xi_{\alpha}a_{l}^{q}a_{\gamma}^{\alpha}c(dx_l)\widetilde{b}_0^1(x_0)c(dx_{\gamma}).\nonumber
\end{align}

It is sufficient to show that
\begin{align}
&{\rm tr} [\pi^+_{\xi_n}\Big(\frac{c[J(\xi)]\widetilde{b}_0^1(x_0)c[J(\xi)]}{(1+\xi_n^2)^2}\Big) \times \partial_{\xi_n}\sigma_{-1}({\widetilde{D}_{W}}^{-1})(x_0)]|_{|\xi'|=1}\\
&=-\frac{\xi_n^2}{8(\xi_n-i)^4(\xi_n+i)^2}h'(0)\sum_{l,\gamma,\eta,\beta=1}^{n}\sum_{\nu,i=1}^{n-1}{\rm tr}[\xi_{i}a_{l}^{n}a_{\gamma}^{n}a_{\nu}^{\eta}a_{\beta}^{i}c(dx_{l})c(dx_{\eta})\overline{c}(dx_{n})\overline{c}(dx_{\nu})c(dx_{\gamma})c(dx_{\beta})]\nonumber\\
&+\frac{\xi_n(1-\xi_n^2)}{16(\xi_n-i)^4(\xi_n+i)^2}h'(0)\sum_{l,\gamma,\eta,\beta=1}^{n}\sum_{\nu=1}^{n-1}{\rm tr}[a_{l}^{n}a_{\gamma}^{n}a_{\nu}^{\eta}a_{\beta}^{n}c(dx_{l})c(dx_{\eta})\overline{c}(dx_{n})\overline{c}(dx_{\nu})c(dx_{\gamma})c(dx_{\beta})]\nonumber\\
&-\frac{\xi_n}{8(\xi_n-i)^4(\xi_n+i)^2}h'(0)\sum_{l,\gamma,\eta,\beta=1}^{n}\sum_{q,\nu,i=1}^{n-1}{\rm tr}[\xi_{q}\xi_{i}a_{l}^{q}a_{\gamma}^{n}a_{\nu}^{\eta}a_{\beta}^{i}c(dx_{l})c(dx_{\eta})\overline{c}(dx_{n})\overline{c}(dx_{\nu})c(dx_{\gamma})c(dx_{\beta})]\nonumber\\
&+\frac{1-\xi_n^2}{16(\xi_n-i)^4(\xi_n+i)^2}h'(0)\sum_{l,\gamma,\eta,\beta=1}^{n}\sum_{q,\nu=1}^{n-1}{\rm tr}[\xi_{q}a_{l}^{q}a_{\gamma}^{n}a_{\nu}^{\eta}a_{\beta}^{n}c(dx_{l})c(dx_{\eta})\overline{c}(dx_{n})\overline{c}(dx_{\nu})c(dx_{\gamma})c(dx_{\beta})]\nonumber\\
&-\frac{\xi_n}{8(\xi_n-i)^4(\xi_n+i)^2}h'(0)\sum_{l,\gamma,\eta,\beta=1}^{n}\sum_{\alpha,\nu,i=1}^{n-1}{\rm tr}[\xi_{\alpha}\xi_{i}a_{l}^{n}a_{\gamma}^{\alpha}a_{\nu}^{\eta}a_{\beta}^{i}c(dx_{l})c(dx_{\eta})\overline{c}(dx_{n})\overline{c}(dx_{\nu})c(dx_{\gamma})c(dx_{\beta})]\nonumber
\end{align}
\begin{align}
&+\frac{1-\xi_n^2}{16(\xi_n-i)^4(\xi_n+i)^2}h'(0)\sum_{l,\gamma,\eta,\beta=1}^{n}\sum_{\alpha,\nu=1}^{n-1}{\rm tr}[\xi_{\alpha}a_{l}^{n}a_{\gamma}^{\alpha}a_{\nu}^{\eta}a_{\beta}^{n}c(dx_{l})c(dx_{\eta})\overline{c}(dx_{n})\overline{c}(dx_{\nu})c(dx_{\gamma})c(dx_{\beta})]\nonumber\\
&+\frac{i\xi_n(i\xi_n+2)}{8(\xi_n-i)^4(\xi_n+i)^2}h'(0)\sum_{l,\gamma,\eta,\beta=1}^{n}\sum_{q,\alpha,\nu,i=1}^{n-1}{\rm tr}[\xi_{q}\xi_{\alpha}\xi_{i}a_{l}^{q}a_{\gamma}^{\alpha}a_{\nu}^{\eta}a_{\beta}^{i}c(dx_{l})c(dx_{\eta})\overline{c}(dx_{n})\overline{c}(dx_{\nu})c(dx_{\gamma})c(dx_{\beta})]\nonumber\\
&-\frac{i(i\xi_n+2)(1-\xi_n^2)}{16(\xi_n-i)^4(\xi_n+i)^2}h'(0)\sum_{l,\gamma,\eta,\beta=1}^{n}\sum_{q,\alpha,\nu=1}^{n-1}{\rm tr}[\xi_{q}\xi_{\alpha}a_{l}^{q}a_{\gamma}^{\alpha}a_{\nu}^{\eta}a_{\beta}^{n}c(dx_{l})c(dx_{\eta})\overline{c}(dx_{n})\overline{c}(dx_{\nu})c(dx_{\gamma})c(dx_{\beta})].\nonumber
\end{align}

By the relation of the Clifford action and ${\rm tr}{[AB]}={\rm tr}{[BA]}$,  we have the equality:
\begin{align}
{\rm tr}[c(dx_{l})c(dx_{\eta})\overline{c}(dx_{n})\overline{c}(dx_{\nu})c(dx_{\gamma})c(dx_{\beta})]=\delta_{l}^{\beta}\delta_{\eta}^{\gamma}\delta_{n}^{\nu}{\rm tr}[\texttt{id}]
-\delta_{l}^{\gamma}\delta_{\eta}^{\beta}\delta_{n}^{\nu}{\rm tr}[\texttt{id}]
+\delta_{l}^{\eta}\delta_{n}^{\nu}\delta_{\gamma}^{\beta}{\rm tr}[\texttt{id}],
\end{align}
in this way
\begin{align}
\sum_{l,\gamma,\eta,\beta=1}^{n}\sum_{\nu=1}^{n-1}{\rm tr}[c(dx_{l})c(dx_{\eta})\overline{c}(dx_{n})\overline{c}(dx_{\nu})c(dx_{\gamma})c(dx_{\beta})]=0.
\end{align}

On account of the above formulas, we have
\begin{align}
-i\int_{|\xi'|=1}\int^{+\infty}_{-\infty}{\rm tr}_{\wedge^*T^*M}[\pi^+_{\xi_n}\Big(\frac{c[J(\xi)]\widetilde{b}_0^1(x_0)c[J(\xi)]}{(1+\xi_n^2)^2}\Big) \times \partial_{\xi_n}\sigma_{-1}({\widetilde{D}_{W}}^{-1})(x_0)]d\xi_n\sigma(\xi')dx'=0.
\end{align}

Similarly, we have
\begin{align}
\pi^+_{\xi_n}\Big(\frac{c[J(\xi)]\overline{c}(V)c[J(\xi)]}{(1+\xi_n^2)^2}\Big)&=-\frac{i\xi_n}{4(\xi_n-i)^2}\sum_{l,\gamma=1}^{n}a_{l}^{n}a_{\gamma}^{n}c(dx_l)\overline{c}(V)c(dx_{\gamma})\nonumber\\
&-\frac{i}{4(\xi_n-i)^2}\sum_{l,\gamma=1}^{n}\sum_{q=1}^{n-1}\xi_{q}a_{l}^{q}a_{\gamma}^{n}c(dx_l)\overline{c}(V)c(dx_{\gamma})\nonumber\\
&-\frac{i}{4(\xi_n-i)^2}\sum_{l,\gamma=1}^{n}\sum_{\alpha=1}^{n-1}\xi_{\alpha}a_{l}^{n}a_{\gamma}^{\alpha}c(dx_l)\overline{c}(V)c(dx_{\gamma})\nonumber\\
&-\frac{i\xi_n+2}{4(\xi_n-i)^2}\sum_{l,\gamma=1}^{n}\sum_{q,\alpha=1}^{n-1}\xi_{q}\xi_{\alpha}a_{l}^{q}a_{\gamma}^{\alpha}c(dx_l)\overline{c}(V)c(dx_{\gamma}).\nonumber
\end{align}

Therefore
\begin{align}
&{\rm tr} [\pi^+_{\xi_n}\Big(\frac{c[J(\xi)]\overline{c}(V)c[J(\xi)]}{(1+\xi_n^2)^2}\Big) \times \partial_{\xi_n}\sigma_{-1}({\widetilde{D}_{W}}^{-1})(x_0)]|_{|\xi'|=1}\\
&=-\frac{\xi_n^2}{2(\xi_n-i)^4(\xi_n+i)^2}h'(0)\sum_{l,\gamma,\beta=1}^{n}\sum_{i=1}^{n-1}{\rm tr}[\xi_{i}a_{l}^{n}a_{\gamma}^{n}a_{\beta}^{i}c(dx_{l})\overline{c}(V)c(dx_{\gamma})c(dx_{\beta})]\nonumber\\
&+\frac{\xi_n(1-\xi_n^2)}{4(\xi_n-i)^4(\xi_n+i)^2}h'(0)\sum_{l,\gamma,\beta=1}^{n}{\rm tr}[a_{l}^{n}a_{\gamma}^{n}a_{\beta}^{n}c(dx_{l})\overline{c}(V)c(dx_{\gamma})c(dx_{\beta})]\nonumber\\
&-\frac{\xi_n}{2(\xi_n-i)^4(\xi_n+i)^2}h'(0)\sum_{l,\gamma,\beta=1}^{n}\sum_{q,i=1}^{n-1}{\rm tr}[\xi_{q}\xi_{i}a_{l}^{q}a_{\gamma}^{n}a_{\beta}^{i}c(dx_{l})\overline{c}(V)c(dx_{\gamma})c(dx_{\beta})]\nonumber\\
&+\frac{1-\xi_n^2}{4(\xi_n-i)^4(\xi_n+i)^2}h'(0)\sum_{l,\gamma,\beta=1}^{n}\sum_{q=1}^{n-1}{\rm tr}[\xi_{q}a_{l}^{q}a_{\gamma}^{n}a_{\beta}^{n}c(dx_{l})\overline{c}(V)c(dx_{\gamma})c(dx_{\beta})]\nonumber\\
&-\frac{\xi_n}{2(\xi_n-i)^4(\xi_n+i)^2}h'(0)\sum_{l,\gamma,\beta=1}^{n}\sum_{\alpha,i=1}^{n-1}{\rm tr}[\xi_{\alpha}\xi_{i}a_{l}^{n}a_{\gamma}^{\alpha}a_{\beta}^{i}c(dx_{l})\overline{c}(V)c(dx_{\gamma})c(dx_{\beta})]\nonumber\\
&+\frac{1-\xi_n^2}{4(\xi_n-i)^4(\xi_n+i)^2}h'(0)\sum_{l,\gamma,\beta=1}^{n}\sum_{\alpha=1}^{n-1}{\rm tr}[\xi_{\alpha}a_{l}^{n}a_{\gamma}^{\alpha}a_{\beta}^{n}c(dx_{l})\overline{c}(V)c(dx_{\gamma})c(dx_{\beta})]\nonumber\\
&+\frac{i\xi_n(i\xi_n+2)}{2(\xi_n-i)^4(\xi_n+i)^2}h'(0)\sum_{l,\gamma,\beta=1}^{n}\sum_{q,\alpha,i=1}^{n-1}{\rm tr}[\xi_{q}\xi_{\alpha}\xi_{i}a_{l}^{q}a_{\gamma}^{\alpha}a_{\beta}^{i}c(dx_{l})\overline{c}(V)c(dx_{\gamma})c(dx_{\beta})]\nonumber\\
&-\frac{i(i\xi_n+2)(1-\xi_n^2)}{4(\xi_n-i)^4(\xi_n+i)^2}h'(0)\sum_{l,\gamma,\beta=1}^{n}\sum_{q,\alpha=1}^{n-1}{\rm tr}[\xi_{q}\xi_{\alpha}a_{l}^{q}a_{\gamma}^{\alpha}a_{\beta}^{n}c(dx_{l})\overline{c}(V)c(dx_{\gamma})c(dx_{\beta})].\nonumber
\end{align}

Applying $c(e_i)c(e_j)+c(e_j)c(e_i)=-2\delta_i^j,$ $\overline{c}(e_i)c(e_j)+c(e_j)\overline{c}(e_i)=0$ and ${\rm tr}{AB}={\rm tr}{BA}$,  we have
\begin{align}
{\rm tr}[c(dx_{l})\overline{c}(V)c(dx_{\gamma})c(dx_{\beta})]=0.
\end{align}

Clearly,
\begin{align}
-i\int_{|\xi'|=1}\int^{+\infty}_{-\infty}{\rm tr}_{\wedge^*T^*M}[\pi^+_{\xi_n}\Big(\frac{c[J(\xi)]\overline{c}(V)c[J(\xi)]}{(1+\xi_n^2)^2}\Big) \times \partial_{\xi_n}\sigma_{-1}({\widetilde{D}_{W}}^{-1})(x_0)]d\xi_n\sigma(\xi')dx'=0.
\end{align}

As in \cite{LW2}, we have
\begin{align}
\Psi_4&=-i\int_{|\xi'|=1}\int^{+\infty}_{-\infty}{\rm tr}_{\wedge^*T^*M}[\pi^+_{\xi_n}\Big(\frac{c[J(\xi)]\widetilde{b}_0^1(x_0)c[J(\xi)]}{(1+\xi_n^2)^2}\Big) \times \partial_{\xi_n}\sigma_{-1}({\widetilde{D}_{W}}^{-1})(x_0)]d\xi_n\sigma(\xi')dx'\\
&-i\int_{|\xi'|=1}\int^{+\infty}_{-\infty}{\rm tr}_{\wedge^*T^*M}[\pi^+_{\xi_n}\Big(\frac{c[J(\xi)]\overline{c}(V)c[J(\xi)]}{(1+\xi_n^2)^2}\Big) \times \partial_{\xi_n}\sigma_{-1}({\widetilde{D}_{W}}^{-1})(x_0)]d\xi_n\sigma(\xi')dx'\nonumber\\
&-i\int_{|\xi'|=1}\int^{+\infty}_{-\infty}{\rm tr}_{\wedge^*T^*M}[\Big(\pi^+_{\xi_n}\Big(\frac{c[J(\xi)]\widetilde{b}_0^2(x_0)c[J(\xi)]}{(1+\xi_n^2)^2}\Big)-h'(0)\pi^+_{\xi_n}\Big(\frac{c[J(\xi)]}{(1+\xi_n^2)^3}c[J(dx_n)]c[J(\xi)]\Big)\nonumber\\
&+\pi^+_{\xi_n}\Big(\frac{c[J(\xi)]}{(1+\xi_n^2)^2}\Big[\sum_{j,p,h=1}^{n}\xi_p\partial_{x_j}(a_{h}^{p})c[J(dx_j)]c(dx_h)+\sum_{p=1}^{n}\sum_{h=1}^{n-1}\xi_pa_{h}^{p}c[J(dx_n)]\partial_{x_n}(c(dx_h))\Big]\Big)\Big)\nonumber\\ &\times \partial_{\xi_n}\sigma_{-1}({\widetilde{D}_{W}}^{-1})(x_0)]d\xi_n\sigma(\xi')dx'\nonumber\\
&=\sum_{l=1}^{n}\sum_{\nu,i=1}^{n-1}(-2(a_{\nu}^{n})^2(a_{l}^{i})^2+2(a_{l}^{i})^2a_{\nu}^{\nu}a_{n}^{n}+2a_{\nu}^{i}a_{l}^{i}a_{\nu}^{n}a_{l}^{n}-2a_{i}^{i}a_{\nu}^{\nu}){\rm tr}[\texttt{id}]\Omega_3h'(0)(\frac{\pi^2}{24})dx'\nonumber\\
&+\sum_{l,j,\beta=1}^{n}\sum_{i=1}^{n-1}\left((a_{\beta}^{i})^2a_{l}^{j}\partial_{x_j}(a_{l}^{n})-a_{l}^{i}a_{\beta}^{j}a_{\beta}^{i}\partial_{x_j}(a_{l}^{n})+a_{l}^{i}a_{l}^{j}a_{\beta}^{i}\partial_{x_j}(a_{\beta}^{n})\right){\rm tr}[\texttt{id}]\Omega_3(-\frac{\pi^2}{12})dx'\nonumber\\
&+\sum_{l,j,\beta=1}^{n}\sum_{i=1}^{n-1}\left(a_{\beta}^{n}a_{l}^{j}a_{\beta}^{i}\partial_{x_j}(a_{l}^{i})-a_{l}^{n}a_{\beta}^{j}a_{\beta}^{i}\partial_{x_j}(a_{l}^{i})+a_{l}^{n}a_{l}^{j}a_{\beta}^{i}\partial_{x_j}(a_{\beta}^{i})\right){\rm tr}[\texttt{id}]\Omega_3(-\frac{\pi^2}{12})dx'\nonumber\\
&+\sum_{l,j,\beta=1}^{n}\sum_{i=1}^{n-1}\left(a_{\beta}^{i}a_{l}^{j}a_{\beta}^{n}\partial_{x_j}(a_{l}^{i})-a_{l}^{i}a_{\beta}^{j}a_{\beta}^{n}\partial_{x_j}(a_{l}^{i})+a_{l}^{i}a_{l}^{j}a_{\beta}^{n}\partial_{x_j}(a_{\beta}^{i})\right){\rm tr}[\texttt{id}]\Omega_3(\frac{\pi^2}{6})dx'\nonumber\\
&+\sum_{l=1}^{n}\sum_{\nu,i=1}^{n-1}(a_{\nu}^{n})^2(a_{l}^{i})^2{\rm tr}[\texttt{id}]\Omega_3h'(0)(-\frac{\pi^2}{24})dx'+\sum_{l=1}^{n}\sum_{\nu,i=1}^{n-1}(a_{\nu}^{i})^2(a_{l}^{n})^2{\rm tr}[\texttt{id}]\Omega_3h'(0)(-\frac{\pi^2}{24})dx'\nonumber\\
&+\sum_{l=1}^{n}\sum_{\nu,i=1}^{n-1}\left(2a_{\nu}^{i}a_{l}^{i}a_{\nu}^{n}a_{l}^{n}-(a_{\nu}^{i})^2(a_{l}^{n})^2\right){\rm tr}[\texttt{id}]\Omega_3h'(0)(\frac{\pi^2}{12})dx'\nonumber\\
&+\sum_{\beta,l=1}^{n}(a_{\beta}^{n})^2(a_{l}^{n})^2{\rm tr}[\texttt{id}]\Omega_3h'(0)(-\frac{\pi}{128})dx'+\sum_{\beta,l=1}^{n}\sum_{i=1}^{n-1}(a_{\beta}^{i})^2(a_{l}^{n})^2{\rm tr}[\texttt{id}]\Omega_3h'(0)(\frac{5\pi^2}{48})dx'\nonumber\\
&+\sum_{\beta,l=1}^{n}\sum_{i=1}^{n-1}\left(2a_{l}^{i}a_{\beta}^{i}a_{l}^{n}a_{\beta}^{n}-(a_{l}^{i})^2(a_{\beta}^{n})^2\right){\rm tr}[\texttt{id}]\Omega_3h'(0)(-\frac{5\pi^2}{32})dx'.\nonumber
\end{align}

{\bf  case (c)}~$r=-1,~l=-2,~k=j=|\alpha|=0$\\

We calculate
\begin{align}
\Psi_5=-i\int_{|\xi'|=1}\int^{+\infty}_{-\infty}{\rm tr}_{\wedge^*T^*M} [\pi^+_{\xi_n}\sigma_{-1}({\widetilde{D}_{W}}^{-1})\times
\partial_{\xi_n}\sigma_{-2}({\widetilde{D}_{W}}^{-1})](x_0)d\xi_n\sigma(\xi')dx'.
\end{align}

Following the same method as (3.38)-(3.40), we can get
\begin{align}
\pi^+_{\xi_n}\left(\frac{1}{1+\xi_{n}^2}\right)(x_0)=\frac{1}{2i(\xi_n-i)};\\
\pi^+_{\xi_n}\left(\frac{\xi_{n}}{1+\xi_{n}^2}\right)(x_0)=\frac{1}{2(\xi_n-i)}.
\end{align}
Consequently,
\begin{align}
\pi^+_{\xi_n}\left(\frac{ic[J(\xi)]}{|\xi|^2}\right)(x_0)|_{|\xi'|=1}
=\frac{1}{2(\xi_{n}-i)}\sum^{n}_{\beta=1}\sum^{n-1}_{i=1}\xi_{i}a_{\beta}^{i}c(dx_{\beta})+\frac{i}{2(\xi_{n}-i)}\sum^{n}_{\beta=1}a_{\beta}^{n}c(dx_{\beta}).
\end{align}

We check at once that
\begin{align}
&\partial_{\xi_n}\sigma_{-2}({\widetilde{D}_{W}}^{-1})(x_0)|_{|\xi'|=1}=\partial_{\xi_n}\Big(\frac{c[J(\xi)]\widetilde{b}_0^2(x_0)c[J(\xi)]}{(1+\xi_n^2)^2}\Big)-h'(0)\partial_{\xi_n}\Big(\frac{c[J(\xi)]}{(1+\xi_n^2)^3}c[J(dx_n)]c[J(\xi)]\Big)\\
&+\partial_{\xi_n}\Big(\frac{c[J(\xi)]}{(1+\xi_n^2)^2}\Big[\sum_{j,p,h=1}^{n}\xi_p\partial_{x_j}(a_{h}^{p})c[J(dx_j)]c(dx_h)+\sum_{p=1}^{n}\sum_{h=1}^{n-1}\xi_pa_{h}^{p}c[J(dx_n)]\partial_{x_n}(c(dx_h))\Big]\Big)\nonumber\\
&+\partial_{\xi_n}\Big(\frac{c[J(\xi)]\widetilde{b}_0^1(x_0)c[J(\xi)]}{(1+\xi_n^2)^2}\Big)+\partial_{\xi_n}\Big(\frac{c[J(\xi)]\overline{c}(V)c[J(\xi)]}{(1+\xi_n^2)^2}\Big).\nonumber
\end{align}

By computation, we have
\begin{align}
\partial_{\xi_n}\Big(\frac{c[J(\xi)]\widetilde{b}_0^1(x_0)c[J(\xi)]}{(1+\xi_n^2)^2}\Big)&=-\frac{2\xi_n(-1+\xi_n^2)}{(1+\xi_n^2)^3}\sum_{l,\gamma=1}^{n}a_{l}^{n}a_{\gamma}^{n}c(dx_l)\widetilde{b}_0^1(x_0)c(dx_{\gamma})\\
&+\frac{1-3\xi_n^2}{(1+\xi_n^2)^3}\sum_{l,\gamma=1}^{n}\sum_{q=1}^{n-1}\xi_{q}a_{l}^{q}a_{\gamma}^{n}c(dx_l)\widetilde{b}_0^1(x_0)c(dx_{\gamma})\nonumber\\
&+\frac{1-3\xi_n^2}{(1+\xi_n^2)^3}\sum_{l,\gamma=1}^{n}\sum_{\alpha=1}^{n-1}\xi_{\alpha}a_{l}^{n}a_{\gamma}^{\alpha}c(dx_l)\widetilde{b}_0^1(x_0)c(dx_{\gamma})\nonumber\\
&-\frac{4\xi_n}{(1+\xi_n^2)^3}\sum_{l,\gamma=1}^{n}\sum_{q,\alpha=1}^{n-1}\xi_{q}\xi_{\alpha}a_{l}^{q}a_{\gamma}^{\alpha}c(dx_l)\widetilde{b}_0^1(x_0)c(dx_{\gamma}),\nonumber
\end{align}
for this reason
\begin{align}
&{\rm tr} [\pi^+_{\xi_n}\sigma_{-1}({\widetilde{D}_{W}}^{-1})(x_0) \times \partial_{\xi_n}\Big(\frac{c[J(\xi)]\widetilde{b}_0^1(x_0)c[J(\xi)]}{(1+\xi_n^2)^2}\Big)]|_{|\xi'|=1}\\
&=-\frac{\xi_n}{2(\xi_n-i)^4(\xi_n+i)^3}h'(0)\sum_{\beta,l,\gamma,\eta=1}^{n}\sum_{i,q,\alpha,\nu=1}^{n-1}{\rm tr}[\xi_{i}\xi_{q}\xi_{\alpha}a_{\beta}^{i}a_{l}^{q}a_{\gamma}^{\alpha}a_{\nu}^{\eta}c(dx_{\beta})c(dx_{l})c(dx_{\eta})\overline{c}(dx_{n})\overline{c}(dx_{\nu})c(dx_{\gamma})]\nonumber\\
&-\frac{i\xi_n}{2(\xi_n-i)^4(\xi_n+i)^3}h'(0)\sum_{\beta,l,\gamma,\eta=1}^{n}\sum_{q,\alpha,\nu=1}^{n-1}{\rm tr}[\xi_{q}\xi_{\alpha}a_{\beta}^{n}a_{l}^{q}a_{\gamma}^{\alpha}a_{\nu}^{\eta}c(dx_{\beta})c(dx_{l})c(dx_{\eta})\overline{c}(dx_{n})\overline{c}(dx_{\nu})c(dx_{\gamma})]\nonumber\\
&+\frac{1-3\xi_n^2}{8(\xi_n-i)^4(\xi_n+i)^3}h'(0)\sum_{\beta,l,\gamma,\eta=1}^{n}\sum_{i,\alpha,\nu=1}^{n-1}{\rm tr}[\xi_{i}\xi_{\alpha}a_{\beta}^{i}a_{l}^{n}a_{\gamma}^{\alpha}a_{\nu}^{\eta}c(dx_{\beta})c(dx_{l})c(dx_{\eta})\overline{c}(dx_{n})\overline{c}(dx_{\nu})c(dx_{\gamma})]\nonumber\\
&+\frac{i(1-3\xi_n^2)}{8(\xi_n-i)^4(\xi_n+i)^3}h'(0)\sum_{\beta,l,\gamma,\eta=1}^{n}\sum_{\alpha,\nu=1}^{n-1}{\rm tr}[\xi_{\alpha}a_{\beta}^{n}a_{l}^{n}a_{\gamma}^{\alpha}a_{\nu}^{\eta}c(dx_{\beta})c(dx_{l})c(dx_{\eta})\overline{c}(dx_{n})\overline{c}(dx_{\nu})c(dx_{\gamma})]\nonumber\\
&+\frac{1-3\xi_n^2}{8(\xi_n-i)^4(\xi_n+i)^3}h'(0)\sum_{\beta,l,\gamma,\eta=1}^{n}\sum_{i,q,\nu=1}^{n-1}{\rm tr}[\xi_{i}\xi_{q}a_{\beta}^{i}a_{l}^{q}a_{\gamma}^{n}a_{\nu}^{\eta}c(dx_{\beta})c(dx_{l})c(dx_{\eta})\overline{c}(dx_{n})\overline{c}(dx_{\nu})c(dx_{\gamma})]\nonumber\\
&+\frac{i(1-3\xi_n^2)}{8(\xi_n-i)^4(\xi_n+i)^3}h'(0)\sum_{\beta,l,\gamma,\eta=1}^{n}\sum_{q,\nu=1}^{n-1}{\rm tr}[\xi_{q}a_{\beta}^{n}a_{l}^{q}a_{\gamma}^{n}a_{\nu}^{\eta}c(dx_{\beta})c(dx_{l})c(dx_{\eta})\overline{c}(dx_{n})\overline{c}(dx_{\nu})c(dx_{\gamma})]\nonumber\\
&-\frac{\xi_n(-1+\xi_n^2)}{4(\xi_n-i)^4(\xi_n+i)^3}h'(0)\sum_{\beta,l,\gamma,\eta=1}^{n}\sum_{i,\nu=1}^{n-1}{\rm tr}[\xi_{i}a_{\beta}^{i}a_{l}^{n}a_{\gamma}^{n}a_{\nu}^{\eta}c(dx_{\beta})c(dx_{l})c(dx_{\eta})\overline{c}(dx_{n})\overline{c}(dx_{\nu})c(dx_{\gamma})]\nonumber\\
&-\frac{i\xi_n(-1+\xi_n^2)}{4(\xi_n-i)^4(\xi_n+i)^3}h'(0)\sum_{\beta,l,\gamma,\eta=1}^{n}\sum_{\nu=1}^{n-1}{\rm tr}[a_{\beta}^{n}a_{l}^{n}a_{\gamma}^{n}a_{\nu}^{\eta}c(dx_{\beta})c(dx_{l})c(dx_{\eta})\overline{c}(dx_{n})\overline{c}(dx_{\nu})c(dx_{\gamma})].\nonumber
\end{align}

Since ${\rm tr}{[AB]}={\rm tr}{[BA]},$ (3.43) shows that
\begin{align}
\sum_{l,\gamma,\eta,\beta=1}^{n}\sum_{\nu=1}^{n-1}{\rm tr}[c(dx_{\beta})c(dx_{l})c(dx_{\eta})\overline{c}(dx_{n})\overline{c}(dx_{\nu})c(dx_{\gamma})]=0,
\end{align}
it is shown that
\begin{align}
&-i\int_{|\xi'|=1}\int^{+\infty}_{-\infty}{\rm tr}_{\wedge^*T^*M} [\pi^+_{\xi_n}\sigma_{-1}({\widetilde{D}_{W}}^{-1})(x_0) \times \partial_{\xi_n}\Big(\frac{c[J(\xi)]\widetilde{b}_0^1(x_0)c[J(\xi)]}{(1+\xi_n^2)^2}\Big)]d\xi_n\sigma(\xi')dx'=0.
\end{align}

Similarly to (3.54) and (3.57), we have
\begin{align}
\partial_{\xi_n}\Big(\frac{c[J(\xi)]\overline{c}(V)c[J(\xi)]}{(1+\xi_n^2)^2}\Big)&=-\frac{2\xi_n(-1+\xi_n^2)}{(1+\xi_n^2)^3}\sum_{l,\gamma=1}^{n}a_{l}^{n}a_{\gamma}^{n}c(dx_l)\overline{c}(V)c(dx_{\gamma})\\
&+\frac{1-3\xi_n^2}{(1+\xi_n^2)^3}\sum_{l,\gamma=1}^{n}\sum_{q=1}^{n-1}\xi_{q}a_{l}^{q}a_{\gamma}^{n}c(dx_l)\overline{c}(V)c(dx_{\gamma})\nonumber\\
&+\frac{1-3\xi_n^2}{(1+\xi_n^2)^3}\sum_{l,\gamma=1}^{n}\sum_{\alpha=1}^{n-1}\xi_{\alpha}a_{l}^{n}a_{\gamma}^{\alpha}c(dx_l)\overline{c}(V)c(dx_{\gamma})\nonumber\\
&-\frac{4\xi_n}{(1+\xi_n^2)^3}\sum_{l,\gamma=1}^{n}\sum_{q,\alpha=1}^{n-1}\xi_{q}\xi_{\alpha}a_{l}^{q}a_{\gamma}^{\alpha}c(dx_l)\overline{c}(V)c(dx_{\gamma}),\nonumber
\end{align}
\begin{align}
&{\rm tr} [\pi^+_{\xi_n}\sigma_{-1}({\widetilde{D}_{W}}^{-1})(x_0) \times \partial_{\xi_n}\Big(\frac{c[J(\xi)]\overline{c}(V)c[J(\xi)]}{(1+\xi_n^2)^2}\Big)]|_{|\xi'|=1}\\
&=-\frac{2\xi_n}{(\xi_n-i)^4(\xi_n+i)^3}h'(0)\sum_{\beta,l,\gamma=1}^{n}\sum_{i,q,\alpha=1}^{n-1}{\rm tr}[\xi_{i}\xi_{q}\xi_{\alpha}a_{\beta}^{i}a_{l}^{q}a_{\gamma}^{\alpha}c(dx_{\beta})c(dx_{l})\overline{c}(V)c(dx_{\gamma})]\nonumber\\
&-\frac{2i\xi_n}{(\xi_n-i)^4(\xi_n+i)^3}h'(0)\sum_{\beta,l,\gamma=1}^{n}\sum_{q,\alpha=1}^{n-1}{\rm tr}[\xi_{q}\xi_{\alpha}a_{\beta}^{n}a_{l}^{q}a_{\gamma}^{\alpha}c(dx_{\beta})c(dx_{l})\overline{c}(V)c(dx_{\gamma})]\nonumber\\
&+\frac{1-3\xi_n^2}{2(\xi_n-i)^4(\xi_n+i)^3}h'(0)\sum_{\beta,l,\gamma=1}^{n}\sum_{i,\alpha=1}^{n-1}{\rm tr}[\xi_{i}\xi_{\alpha}a_{\beta}^{i}a_{l}^{n}a_{\gamma}^{\alpha}c(dx_{\beta})c(dx_{l})\overline{c}(V)c(dx_{\gamma})]\nonumber\\
&+\frac{i(1-3\xi_n^2)}{2(\xi_n-i)^4(\xi_n+i)^3}h'(0)\sum_{\beta,l,\gamma=1}^{n}\sum_{\alpha=1}^{n-1}{\rm tr}[\xi_{\alpha}a_{\beta}^{n}a_{l}^{n}a_{\gamma}^{\alpha}c(dx_{\beta})c(dx_{l})\overline{c}(V)c(dx_{\gamma})]\nonumber\\
&+\frac{1-3\xi_n^2}{2(\xi_n-i)^4(\xi_n+i)^3}h'(0)\sum_{\beta,l,\gamma=1}^{n}\sum_{i,q=1}^{n-1}{\rm tr}[\xi_{i}\xi_{q}a_{\beta}^{i}a_{l}^{q}a_{\gamma}^{n}c(dx_{\beta})c(dx_{l})\overline{c}(V)c(dx_{\gamma})]\nonumber\\
&+\frac{i(1-3\xi_n^2)}{2(\xi_n-i)^4(\xi_n+i)^3}h'(0)\sum_{\beta,l,\gamma=1}^{n}\sum_{q=1}^{n-1}{\rm tr}[\xi_{q}a_{\beta}^{n}a_{l}^{q}a_{\gamma}^{n}c(dx_{\beta})c(dx_{l})\overline{c}(V)c(dx_{\gamma})]\nonumber\\
&-\frac{\xi_n(-1+\xi_n^2)}{(\xi_n-i)^4(\xi_n+i)^3}h'(0)\sum_{\beta,l,\gamma=1}^{n}\sum_{i=1}^{n-1}{\rm tr}[\xi_{i}a_{\beta}^{i}a_{l}^{n}a_{\gamma}^{n}c(dx_{\beta})c(dx_{l})\overline{c}(V)c(dx_{\gamma})]\nonumber\\
&-\frac{i\xi_n(-1+\xi_n^2)}{(\xi_n-i)^4(\xi_n+i)^3}h'(0)\sum_{\beta,l,\gamma=1}^{n}{\rm tr}[a_{\beta}^{n}a_{l}^{n}a_{\gamma}^{n}c(dx_{\beta})c(dx_{l})\overline{c}(V)c(dx_{\gamma})]\nonumber
\end{align}
and
\begin{align}
&-i\int_{|\xi'|=1}\int^{+\infty}_{-\infty}{\rm tr}_{\wedge^*T^*M} [\pi^+_{\xi_n}\sigma_{-1}({\widetilde{D}_{W}}^{-1})(x_0) \times \partial_{\xi_n}\Big(\frac{c[J(\xi)]\overline{c}(V)c[J(\xi)]}{(1+\xi_n^2)^2}\Big)]d\xi_n\sigma(\xi')dx'=0.
\end{align}

We can rewrite (3.49) as
\begin{align}
\Psi_5&=-i\int_{|\xi'|=1}\int^{+\infty}_{-\infty}{\rm tr}_{\wedge^*T^*M}[\pi^+_{\xi_n}\sigma_{-1}({\widetilde{D}_{W}}^{-1})(x_0) \times \partial_{\xi_n}\Big(\frac{c[J(\xi)]\widetilde{b}_0^1(x_0)c[J(\xi)]}{(1+\xi_n^2)^2}\Big)]d\xi_n\sigma(\xi')dx'\\
&-i\int_{|\xi'|=1}\int^{+\infty}_{-\infty}{\rm tr}_{\wedge^*T^*M}[\pi^+_{\xi_n}\sigma_{-1}({\widetilde{D}_{W}}^{-1})(x_0) \times \partial_{\xi_n}\Big(\frac{c[J(\xi)]\overline{c}(V)c[J(\xi)]}{(1+\xi_n^2)^2}\Big)]d\xi_n\sigma(\xi')dx'\nonumber\\
&-i\int_{|\xi'|=1}\int^{+\infty}_{-\infty}{\rm tr}_{\wedge^*T^*M}[\pi^+_{\xi_n}\sigma_{-1}({\widetilde{D}_{W}}^{-1})(x_0) \times \Big(\partial_{\xi_n}\Big(\frac{c[J(\xi)]\widetilde{b}_0^2(x_0)c[J(\xi)]}{(1+\xi_n^2)^2}\Big)\nonumber\\
&-h'(0)\partial_{\xi_n}\Big(\frac{c[J(\xi)]}{(1+\xi_n^2)^3}c[J(dx_n)]c[J(\xi)]\Big)+\partial_{\xi_n}\Big(\frac{c[J(\xi)]}{(1+\xi_n^2)^2}\Big[\sum_{j,p,h=1}^{n}\xi_p\partial_{x_j}(a_{h}^{p})c[J(dx_j)]c(dx_h)\nonumber\\
&+\sum_{p=1}^{n}\sum_{h=1}^{n-1}\xi_pa_{h}^{p}c[J(dx_n)]\partial_{x_n}(c(dx_h))\Big]\Big)\Big)]d\xi_n\sigma(\xi')dx'\nonumber\\
&=\sum_{l=1}^{n}\sum_{\nu,i=1}^{n-1}(-2(a_{\nu}^{n})^2(a_{l}^{i})^2+2(a_{l}^{i})^2a_{\nu}^{\nu}a_{n}^{n}+2a_{\nu}^{i}a_{l}^{i}a_{\nu}^{n}a_{l}^{n}-2a_{i}^{i}a_{\nu}^{\nu}){\rm tr}[\texttt{id}]\Omega_3h'(0)(-\frac{\pi^2}{24})dx'\nonumber\\
&+\sum_{l,j,\beta=1}^{n}\sum_{i=1}^{n-1}\left((a_{\beta}^{i})^2a_{l}^{j}\partial_{x_j}(a_{l}^{n})-a_{l}^{i}a_{\beta}^{j}a_{\beta}^{i}\partial_{x_j}(a_{l}^{n})+a_{l}^{i}a_{l}^{j}a_{\beta}^{i}\partial_{x_j}(a_{\beta}^{n})\right){\rm tr}[\texttt{id}]\Omega_3(\frac{\pi^2}{12})dx'\nonumber\\
&+\sum_{l,j,\beta=1}^{n}\sum_{i=1}^{n-1}\left(a_{\beta}^{n}a_{l}^{j}a_{\beta}^{i}\partial_{x_j}(a_{l}^{i})-a_{l}^{n}a_{\beta}^{j}a_{\beta}^{i}\partial_{x_j}(a_{l}^{i})+a_{l}^{n}a_{l}^{j}a_{\beta}^{i}\partial_{x_j}(a_{\beta}^{i})\right){\rm tr}[\texttt{id}]\Omega_3(\frac{\pi^2}{12})dx'\nonumber\\
&+\sum_{l,j,\beta=1}^{n}\sum_{i=1}^{n-1}\left(a_{\beta}^{i}a_{l}^{j}a_{\beta}^{n}\partial_{x_j}(a_{l}^{i})-a_{l}^{i}a_{\beta}^{j}a_{\beta}^{n}\partial_{x_j}(a_{l}^{i})+a_{l}^{i}a_{l}^{j}a_{\beta}^{n}\partial_{x_j}(a_{\beta}^{i})\right){\rm tr}[\texttt{id}]\Omega_3(-\frac{\pi^2}{6})dx'\nonumber\\
&+\sum_{l=1}^{n}\sum_{\nu,i=1}^{n-1}(a_{\nu}^{n})^2(a_{l}^{i})^2{\rm tr}[\texttt{id}]\Omega_3h'(0)(\frac{\pi^2}{24})dx'+\sum_{l=1}^{n}\sum_{\nu,i=1}^{n-1}(a_{\nu}^{i})^2(a_{l}^{n})^2{\rm tr}[\texttt{id}]\Omega_3h'(0)(\frac{\pi^2}{24})dx'\nonumber\\
&+\sum_{l=1}^{n}\sum_{\nu,i=1}^{n-1}\left(2a_{\nu}^{i}a_{l}^{i}a_{\nu}^{n}a_{l}^{n}-(a_{\nu}^{i})^2(a_{l}^{n})^2\right){\rm tr}[\texttt{id}]\Omega_3h'(0)(-\frac{\pi^2}{12})dx'\nonumber
\end{align}
\begin{align}
&+\sum_{\beta,l=1}^{n}(a_{\beta}^{n})^2(a_{l}^{n})^2{\rm tr}[\texttt{id}]\Omega_3h'(0)(\frac{\pi}{128})dx'+\sum_{\beta,l=1}^{n}\sum_{i=1}^{n-1}(a_{\beta}^{i})^2(a_{l}^{n})^2{\rm tr}[\texttt{id}]\Omega_3h'(0)(-\frac{5\pi^2}{48})dx'\nonumber\\
&+\sum_{\beta,l=1}^{n}\sum_{i=1}^{n-1}\left(2a_{l}^{i}a_{\beta}^{i}a_{l}^{n}a_{\beta}^{n}-(a_{l}^{i})^2(a_{\beta}^{n})^2\right){\rm tr}[\texttt{id}]\Omega_3h'(0)(\frac{5\pi^2}{32})dx'.\nonumber
\end{align}

In summary,
\begin{align}
\Psi&=\Psi_1+\Psi_2+\Psi_3+\Psi_4+\Psi_5\\
&=\sum_{\beta=1}^{n}\sum_{i=1}^{n-1}a_{\beta}^{i}\partial_{x_i}(a_{\beta}^{n}){\rm tr}[\texttt{id}]\Omega_3(-\frac{\pi}{8}+\frac{\pi^2}{3})dx'+\sum_{\beta=1}^{n}\sum_{i=1}^{n-1}a_{\beta}^{n}\partial_{x_i}(a_{\beta}^{i}){\rm tr}[\texttt{id}]\Omega_3(-\frac{\pi^2}{6})dx'.\nonumber
\end{align}

Applying Lemma 3.8 in \cite{LW2}, we see that
\begin{lem}{\rm \cite{LW2}}
\begin{align}
\sum_{\beta=1}^{n}\sum_{i=1}^{n-1}a_{\beta}^{i}\partial_{x_i}(a_{\beta}^{n})=\sum_{\beta=1}^{n}\langle\nabla_{J(e_{\beta})}^{L}(Je_{n}), e_{\beta}\rangle-\sum_{\beta=1}^{n}g^{M}\left(J(\frac{\partial}{\partial{x_{n}}}), \frac{\partial}{\partial{x_{n}}}\right)\langle\nabla_{J(e_{\beta})}^{L}(\frac{\partial}{\partial{x_{n}}}), e_{\beta}\rangle,
\end{align}
\begin{align}
\sum_{\beta=1}^{n}\sum_{i=1}^{n-1}a_{\beta}^{n}\partial_{x_i}(a_{\beta}^{i})=-\sum_{\beta=1}^{n}\langle\nabla_{J(e_{\beta})}^{L}(Je_{n}), e_{\beta}\rangle+\sum_{\beta=1}^{n}g^{M}\left(J(\frac{\partial}{\partial{x_{n}}}), \frac{\partial}{\partial{x_{n}}}\right)\langle\nabla_{J(e_{\beta})}^{L}(\frac{\partial}{\partial{x_{n}}}), e_{\beta}\rangle.
\end{align}
\end{lem}

We can now formulate our main results in this section.
\begin{thm}
Let $M$ be a $4$-dimensional almost product Riemannian spin manifold with the boundary $\partial M$ and the metric $g^M$ as above, $\widetilde{D}_{W}$ be $J$-Witten deformation on $M$, then
\begin{align}
\widetilde{{\rm Wres}}[\pi^+{\widetilde{D}_{W}}^{-1}\circ\pi^+{\widetilde{D}_{W}}^{-1}]=&8\pi^{2}\int_{M}\Big(-\frac{5}{3}s-4|V|^2-2\sum_{\nu,j=1}^{4}g^{M}(\nabla_{e_{j}}^{L}(J)e_{\nu}, (\nabla^{L}_{e_{\nu}}J)e_{j})\\
&-2\sum_{\nu,j=1}^{4}g^{M}(J(e_{\nu}), (\nabla^{L}_{e_{j}}(\nabla^{L}_{e_{\nu}}(J)))e_{j}-(\nabla^{L}_{\nabla^{L}_{e_{j}}e_{\nu}}(J))e_{j})\nonumber\\
&-\sum_{\alpha,\nu,j=1}^{4}g^{M}(J(e_{\alpha}), (\nabla^{L}_{e_{\nu}}J)e_{j})g^{M}((\nabla^{L}_{e_{\alpha}}J)e_{j}, J(e_{\nu}))\nonumber\\
&-\sum_{\alpha,\nu,j=1}^{4}g^{M}(J(e_{\alpha}), (\nabla^{L}_{e_{\alpha}}J)e_{j})g^{M}(J(e_{\nu}), (\nabla^{L}_{e_{\nu}}J)e_{j})\nonumber
\end{align}
\begin{align}
&+\sum_{\nu,j=1}^{4}g^{M}((\nabla^{L}_{e_{\nu}}J)e_{j}, (\nabla^{L}_{e_{\nu}}J)e_{j}))\Big)d{\rm Vol_{M} }\nonumber\\
&+\int_{\partial M}(-2\pi+8\pi^2)\Big(\sum_{\beta=1}^{4}\langle\nabla_{J(e_{\beta})}^{L}(Je_{4}), e_{\beta}\rangle\nonumber\\
&-\sum_{\beta=1}^{4}g^{M}\left(J(\frac{\partial}{\partial{x_{4}}}), \frac{\partial}{\partial{x_{4}}}\right)\langle\nabla_{J(e_{\beta})}^{L}(\frac{\partial}{\partial{x_{4}}}), e_{\beta}\rangle\Big)\Omega_3d{\rm Vol_{\partial M}},\nonumber
\end{align}
where $s$ is the scalar curvature.
\end{thm}

\section{ The Kastler-Kalau-Walze type theorem for $6$-dimensional manifolds with boundary }

In this section, we develop the Kastler-Kalau-Walze type theorem associated with $J$-Witten deformation to six-dimensional almost product Riemannian spin manifold with (respectively without) boundary.
From (3.9), we know that
\begin{equation}
\widetilde{{\rm Wres}}[\pi^+{\widetilde{D}_{W}}^{-1}\circ\pi^+{\widetilde{D}_{W}}^{-3}]=\int_M\int_{|\xi|=1}{\rm tr}_{\wedge^*T^*M}[\sigma_{-6}({\widetilde{D}_{W}}^{-4})]\sigma(\xi)dx+\int_{\partial M}\overline{\Psi},
\end{equation}
where $\widetilde{{\rm Wres}}$ denote noncommutative residue on minifolds with boundary,
\begin{align}
\overline{\Psi}&=\int_{|\xi'|=1}\int^{+\infty}_{-\infty}\sum^{\infty}_{j, k=0}\sum\frac{(-i)^{|\alpha|+j+k+1}}{\alpha!(j+k+1)!}
\times {\rm tr}_{{\wedge^*T^*M}}[\partial^j_{x_n}\partial^\alpha_{\xi'}\partial^k_{\xi_n}\sigma^+_{r}({\widetilde{D}_{W}}^{-1})(x',0,\xi',\xi_n)\\
&\times\partial^\alpha_{x'}\partial^{j+1}_{\xi_n}\partial^k_{x_n}\sigma_{l}({\widetilde{D}_{W}}^{-3})(x',0,\xi',\xi_n)]d\xi_n\sigma(\xi')dx'\nonumber
\end{align}
and the sum is taken over $r+\ell-k-j-|\alpha|-1=-6, \ r\leq-1, \ell\leq -3$.

By Theorem 2.2, we compute the interior term of (4.1)
\begin{align}
&\int_M\int_{|\xi|=1}{\rm tr}_{\wedge^*T^*M}[\sigma_{-6}({\widetilde{D}_{W}}^{-4})]\sigma(\xi)dx=32\pi^{3}\\
&\int_{M}\Big(-\frac{5}{3}s-4|V|^2
-2\sum_{\nu,j=1}^{6}g^{M}(\nabla_{e_{j}}^{L}(J)e_{\nu}, (\nabla^{L}_{e_{\nu}}J)e_{j})\nonumber\\
&-2\sum_{\nu,j=1}^{6}g^{M}(J(e_{\nu}), (\nabla^{L}_{e_{j}}(\nabla^{L}_{e_{\nu}}(J)))e_{j}-(\nabla^{L}_{\nabla^{L}_{e_{j}}e_{\nu}}(J))e_{j})\nonumber\\
&-\sum_{\alpha,\nu,j=1}^{6}g^{M}(J(e_{\alpha}), (\nabla^{L}_{e_{\nu}}J)e_{j})g^{M}((\nabla^{L}_{e_{\alpha}}J)e_{j}, J(e_{\nu}))\nonumber
\end{align}
\begin{align}
&-\sum_{\alpha,\nu,j=1}^{6}g^{M}(J(e_{\alpha}), (\nabla^{L}_{e_{\alpha}}J)e_{j})g^{M}(J(e_{\nu}), (\nabla^{L}_{e_{\nu}}J)e_{j})\nonumber\\
&+\sum_{\nu,j=1}^{6}g^{M}((\nabla^{L}_{e_{\nu}}J)e_{j}, (\nabla^{L}_{e_{\nu}}J)e_{j})\Big)d{\rm Vol_{M} }.\nonumber
\end{align}

So we only need to compute $\int_{\partial M} \overline{\Psi}$. Let us now turn to compute the specification of ${\widetilde{D}_{W}}^3$.
\begin{align}
{\widetilde{D}_{W}}^{3}&=\sum^{n}_{i=1}c[J(e_{i})]\langle e_{i},dx_{l}\rangle(-g^{ij}\partial_{l}\partial_{i}\partial_{j})+\sum^{n}_{i=1}c[J(e_{i})]\langle e_{i},dx_{l}\rangle\Big(-(\partial_{l}g^{ij})\partial_{i}\partial_{j}-g^{ij}(4\sigma_{i}\partial_{j}+4a_{i}\partial_{j}\nonumber\\
&-2\Gamma^{k}_{ij}\partial_{k})\partial_{l}+2\sum^{n}_{\alpha,\beta,\gamma=1}c[J(e_{\alpha})]c[(\nabla^{L}_{e_{\alpha}}J)e_{\beta}]\langle e_{\beta}, dx_{\gamma}\rangle\partial_{l}\partial_{\gamma}\Big)+\Big(\sum^{n}_{i=1}c[J(e_{i})](\sigma_{i}+a_{i})+\overline{c}(V)\Big)\nonumber\\
&\times(-g^{ij}\partial_{i}\partial_{j})+\sum^{n}_{i=1}c[J(e_{i})]\langle e_{i},dx_{l}\rangle\Big[-2(\partial_{l}g^{ij})\sigma_{i}\partial_{j}-2g^{ij}(\partial_{l}\sigma_{i})\partial_{j}-2(\partial_{l}g^{ij})a_{i}\partial_{j}\nonumber\\
&-2g^{ij}(\partial_{l}a_{i})\partial_{j}+(\partial_{l}g^{ij})\Gamma^{k}_{ij}\partial_{k}+g^{ij}(\partial_{l}\Gamma^{k}_{ij})\partial_{k}+\sum^{n}_{\alpha,\beta,\gamma=1}\partial_{l}\Big(c[J(e_{\alpha})]c[(\nabla^{L}_{e_{\alpha}}J)e_{\beta}]\Big)\langle e_{\beta}, dx_{\gamma}\rangle\partial_{\gamma}\nonumber\\
&+\sum^{n}_{\alpha,\beta,\gamma=1}c[J(e_{\alpha})]c[(\nabla^{L}_{e_{\alpha}}J)e_{\beta}]\Big(\partial_{l}\langle e_{\beta}, dx_{\gamma}\rangle\Big)\partial_{\gamma}\Big]+\sum^{n}_{i=1}c[J(e_{i})]\langle e_{i},dx_{l}\rangle\partial_{l}\Big(-g^{ij}((\partial_{i}\sigma_{j})\nonumber\\
&+(\partial_{i}a_{j})+\sigma_{i}\sigma_{j}+\sigma_{i}a_{j}+a_{i}\sigma_{j}+a_{i}a_{j}-\Gamma^{k}_{ij}\sigma_{k}-\Gamma^{k}_{ij}a_{k})+\sum^{n}_{\alpha,\beta,\gamma=1}c[J(e_{\alpha})]c[(\nabla^{L}_{e_{\alpha}}J)e_{\beta}]\langle e_{\beta}, dx_{\gamma}\rangle\nonumber\\
&\times(\sigma_{\gamma}+a_{\gamma})-\frac{1}{8}\sum_{i,j,k,l=1}^{n}R(J(e_{i}), J(e_{j}), e_{k}, e_{l})\overline{c}(e_{i})\overline{c}(e_{j})c(e_{k})c(e_{l})+\frac{1}{4}s+\sum^{n}_{i=1}c[J(e_{i})]\overline{c}(\nabla^{L}_{e_{i}}V)\nonumber\\
&+|V|^{2}\Big)+\Big(\sum^{n}_{i=1}c[J(e_{i})](\sigma_{i}+a_{i})+\overline{c}(V)\Big)\Big(\frac{1}{4}s-2\sigma^{j}\partial_{j}-2a^{j}\partial_{j}+\Gamma^{k}\partial_{k}-g^{ij}((\partial_{i}\sigma_{j})+(\partial_{i}a_{j})\nonumber\\
&+\sigma_{i}\sigma_{j}+\sigma_{i}a_{j}+a_{i}\sigma_{j}+a_{i}a_{j}-\Gamma^{k}_{ij}\sigma_{k}-\Gamma^{k}_{ij}a_{k})+\sum^{n}_{\alpha,\beta,\gamma=1}c[J(e_{\alpha})]c[(\nabla^{L}_{e_{\alpha}}J)e_{\beta}]\langle e_{\beta}, dx_{\gamma}\rangle(\partial_{\gamma}+\sigma_{\gamma}\nonumber\\
&+a_{\gamma})-\frac{1}{8}\sum_{i,j,k,l=1}^{n}R(J(e_{i}), J(e_{j}), e_{k}, e_{l})\overline{c}(e_{i})\overline{c}(e_{j})c(e_{k})c(e_{l})+\sum^{n}_{i=1}c[J(e_{i})]\overline{c}(\nabla^{L}_{e_{i}}V)+|V|^{2}\Big).\nonumber
\end{align}
Then, we obtain
\begin{lem} The following identities hold:
\begin{align}
\sigma_3({\widetilde{D}_{W}}^{3})&=ic[J(\xi)]|\xi|^{2};\\
\sigma_2({\widetilde{D}_{W}}^{3})&=\sum^{n}_{i,j,l=1}c[J(dx_{l})]\partial_{l}(g^{ij})\xi_{i}\xi_{j}+c[J(\xi)](4\sigma^k+4a^k-2\Gamma^k)\xi_{k}\\
&-2\sum^{n}_{\alpha=1}c[J(\xi)]c[J(e_{\alpha})]c[(\nabla^{L}_{e_{\alpha}}J)(\xi^{*})]+\frac{1}{4}|\xi|^2\sum^{n}_{s,t,l=1}\omega_{s,t}(e_{l})c[J(e_{l})]\overline{c}(e_{s})\overline{c}(e_{t})\nonumber\\
&-\frac{1}{4}|\xi|^2\sum^{n}_{s,t,l=1}\omega_{s,t}(e_{l})c[J(e_{l})]c(e_{s})c(e_{t})+|\xi|^2\overline{c}(V),\nonumber
\end{align}
where $\xi^{*}=\sum^{n}_{\beta=1}\langle e_{\beta}, \xi\rangle e_{\beta}.$
\end{lem}

Suppose that
\begin{eqnarray}
\sigma({\widetilde{D}_{W}}^3)=p_3+p_2+p_1+p_0;
~\sigma({\widetilde{D}_{W}}^{-3})=\sum^{\infty}_{j=3}q_{-j}.
\end{eqnarray}
Then
\begin{align}
1=\sigma({\widetilde{D}_{W}}^3\circ {\widetilde{D}_{W}}^{-3})&=
\sum_{\alpha}\frac{1}{\alpha!}\partial^{\alpha}_{\xi}
[\sigma({\widetilde{D}_{W}}^3)]{{D}}^{\alpha}_{x}
[\sigma({\widetilde{D}_{W}}^{-3})] \\
&=(p_3+p_2+p_1+p_0)(q_{-3}+q_{-4}+q_{-5}+\cdots) \nonumber\\
&+\sum_j(\partial_{\xi_j}p_3+\partial_{\xi_j}p_2+\partial_{\xi_j}p_1+\partial_{\xi_j}p_0)
(D_{x_j}q_{-3}+D_{x_j}q_{-4}+D_{x_j}q_{-5}+\cdots) \nonumber\\
&=p_3q_{-3}+(p_3q_{-4}+p_2q_{-3}+\sum_j\partial_{\xi_j}p_3D_{x_j}q_{-3})+\cdots,\nonumber
\end{align}
and consequently
\begin{equation}
q_{-3}=p_3^{-1};~q_{-4}=-p_3^{-1}[p_2p_3^{-1}+\sum_j\partial_{\xi_j}p_3D_{x_j}(p_3^{-1})].
\end{equation}

\begin{lem} The following identities hold:
\begin{align}
\sigma_{-3}({\widetilde{D}_{W}}^{-3})&=\frac{ic[J(\xi)]}{|\xi|^{4}};\\
\sigma_{-4}({\widetilde{D}_{W}}^{-3})&=\frac{c[J(\xi)]\sigma_{2}({\widetilde{D}_{W}}^{3})c[J(\xi)]}{|\xi|^{8}}+\frac{c[J(\xi)]}{|\xi|^{10}}\sum_ {j=1}^{n} \Big(c[J(dx_j)]|\xi|^{2}+2\xi_{j}c[J(\xi)]\Big)\Big[\partial_{x_j}(c[J(\xi)])|\xi|^2\\
&-2c[J(\xi)]\partial_{x_j}(|\xi|^2)\Big].\nonumber
\end{align}
\end{lem}

When $n=6$, then ${\rm tr}_{\wedge^*T^*M}[{\rm \texttt{id}}]=64.$
Since the sum is taken over $r+\ell-k-j-|\alpha|-1=-6, \ r\leq-1, \ell\leq -3$, then we have the $\int_{{\partial}{M}}\overline{\Psi}$ is the sum of the following five cases:\\

{\bf case (a)~(I)}~$r=-1, l=-3, j=k=0, |\alpha|=1$\\

By (4.2), we compute that
\begin{equation}
\overline{\Psi}_1=-\int_{|\xi'|=1}\int^{+\infty}_{-\infty}\sum_{|\alpha|=1}{\rm tr}_{{\wedge^*T^*M}}
[\partial^{\alpha}_{\xi'}\pi^{+}_{\xi_{n}}\sigma_{-1}({\widetilde{D}_{W}}^{-1})\times\partial^{\alpha}_{x'}\partial_{\xi_{n}}\sigma_{-3}({\widetilde{D}_{W}}^{-3})](x_0)d\xi_n\sigma(\xi')dx'.
\end{equation}

{\bf case (a)~(II)}~$r=-1, l=-3, |\alpha|=k=0, j=1$\\

It is easy to check that
\begin{equation}
\overline{\Psi}_2=-\frac{1}{2}\int_{|\xi'|=1}\int^{+\infty}_{-\infty} {\rm tr}_{{\wedge^*T^*M}}[\partial_{x_{n}}\pi^{+}_{\xi_{n}}\sigma_{-1}({\widetilde{D}_{W}}^{-1})\times\partial^{2}_{\xi_{n}}\sigma_{-3}({\widetilde{D}_{W}}^{-3})](x_0)d\xi_n\sigma(\xi')dx'.
\end{equation}

{\bf case (a)~(III)}~$r=-1,l=-3,|\alpha|=j=0,k=1$\\

We notice that
\begin{align}
\overline{\Psi}_3=-\frac{1}{2}\int_{|\xi'|=1}\int^{+\infty}_{-\infty}{\rm tr}_{{\wedge^*T^*M}}[\partial_{\xi_{n}}\pi^{+}_{\xi_{n}}\sigma_{-1}({\widetilde{D}_{W}}^{-1})
      \times\partial_{\xi_{n}}\partial_{x_{n}}\sigma_{-3}({\widetilde{D}_{W}}^{-3})](x_0)d\xi_n\sigma(\xi')dx'.
\end{align}

By (3.31)-(3.66) in \cite{LW3}, we obtain
\begin{align}
\overline{\Psi}_1+\overline{\Psi}_2+\overline{\Psi}_3&=
\sum_{l,j=1}^{n}a_{l}^{j}\partial_{x_j}(a_{l}^{n}){\rm tr}[\texttt{id}]\Omega_4(-\frac{\pi}{16}+\frac{\pi^3}{6})dx'\nonumber\\
&+\sum_{l=1}^{n}\sum_{i=1}^{n-1}(a_{l}^{i})^2{\rm tr}[\texttt{id}]\Omega_4h'(0)(\frac{7 \pi ^3}{240})dx'\nonumber\\
&+\sum_{l=1}^{n}(a_{l}^{n})^2{\rm tr}[\texttt{id}]\Omega_4h'(0)(\frac{3 \pi }{128})dx',\nonumber
\end{align}
where ${\rm \Omega_{4}}$ is the canonical volume of $S^{4}.$\\

{\bf case (b)}~$r=-1,l=-4,|\alpha|=j=k=0$\\

Using (4.2), we get
\begin{align}
\overline{\Psi}_4&=-i\int_{|\xi'|=1}\int^{+\infty}_{-\infty}{\rm tr}_{{\wedge^*T^*M}}[\pi^{+}_{\xi_{n}}\sigma_{-1}({\widetilde{D}_{W}}^{-1})
      \times\partial_{\xi_{n}}\sigma_{-4}({\widetilde{D}_{W}}^{-3})](x_0)d\xi_n\sigma(\xi')dx'\\
&=i\int_{|\xi'|=1}\int^{+\infty}_{-\infty}{\rm tr}_{{\wedge^*T^*M}} [\partial_{\xi_n}\pi^+_{\xi_n}\sigma_{-1}({\widetilde{D}_{W}}^{-1})\times
\sigma_{-4}({\widetilde{D}_{W}}^{-3})](x_0)d\xi_n\sigma(\xi')dx'.\nonumber
\end{align}

We can assert that
\begin{align}
\pi^+_{\xi_n}\partial_{\xi_n}\left(\frac{ic[J(\xi)]}{|\xi|^2}\right)(x_0)|_{|\xi'|=1}
&=-\frac{1}{2(\xi_{n}-i)^2}\sum^{n}_{l=1}\sum^{n-1}_{i=1}\xi_{i}a_{l}^ic(dx_{l})-\frac{i}{2(\xi_{n}-i)^2}\sum^{n}_{l=1}a_{l}^nc(dx_{l}).
\end{align}
For simplicity of notation, we write
\begin{align}
\widetilde{B}_1(x_0)&=-\frac{1}{2(\xi_{n}-i)^2}\sum^{n}_{l=1}\sum^{n-1}_{i=1}\xi_{i}a_{l}^ic(dx_{l});\\
\widetilde{B}_2(x_0)&=-\frac{i}{2(\xi_{n}-i)^2}\sum^{n}_{l=1}a_{l}^nc(dx_{l}).
\end{align}

By computations, we have
\begin{align}
\sigma_{-4}({{D}_{J}}^{-3})(x_{0})|_{|\xi'|=1}&=\frac{1}{(1+\xi_{n}^{2})^{3}}h'(0)\sum_{\eta, \Gamma, \Omega, \Lambda, \Pi=1}^{n}\xi_{\Gamma}\xi_{\Omega}a_{\Lambda}^{\Gamma}a_{\eta}^{n}a_{\Pi}^{\Omega}c(dx_{\Lambda})c(dx_{\eta})c(dx_{\Pi})\\
&-\frac{1}{(1+\xi_{n}^{2})^{3}}h'(0)\sum_{\chi,\tau=1}^{n}\sum_{\gamma=1}^{n-1}\xi_{\gamma}\xi_{\chi}a_{\tau}^{\chi}c(dx_{\gamma})c(dx_{n})c(dx_{\tau})\nonumber\\
&+\frac{1}{(1+\xi_{n}^{2})^{3}}h'(0)\sum_{\chi,\tau=1}^{n}\sum_{\gamma=1}^{n-1}\xi_{\gamma}\xi_{\chi}a_{\tau}^{\chi}\overline{c}(dx_{\gamma})\overline{c}(dx_{n})c(dx_{\tau})\nonumber\\
&+\frac{5}{(1+\xi_{n}^{2})^{3}}h'(0)\sum_{\rho,\theta=1}^{n}\xi_{n}\xi_{\rho}a_{\theta}^{\rho}c(dx_{\theta})\nonumber\\
&+\frac{2}{(1+\xi_{n}^{2})^{3}}\sum_{\alpha,\beta,\lambda,\omega=1}^{n}\xi_{\lambda}a_{\alpha}^{\beta}a_{\omega}^{\lambda}c(dx_{\beta})c[(\nabla^{L}_{e_{\alpha}}J)(\xi^{*})]c(dx_{\omega})\nonumber\\
&-\frac{1}{4(1+\xi_{n}^{2})^{3}}h'(0)\sum_{\eta, \Phi, b, \Psi, c=1}^{n}\sum_{\nu=1}^{n-1}\xi_{\Phi}\xi_{b}a_{\nu}^{\eta}a_{\Psi}^{\Phi}a_{c}^{b}c(dx_{\Psi})c(dx_{\eta})c(dx_{n})c(dx_{\nu})c(dx_{c})\nonumber\\
&+\frac{1}{4(1+\xi_{n}^{2})^{3}}h'(0)\sum_{\eta, \Phi, b, \Psi, c=1}^{n}\sum_{\nu=1}^{n-1}\xi_{\Phi}\xi_{b}a_{\nu}^{\eta}a_{\Psi}^{\Phi}a_{c}^{b}c(dx_{\Psi})c(dx_{\eta})\overline{c}(dx_{n})\overline{c}(dx_{\nu})c(dx_{c})\nonumber\\
&+\frac{1}{(1+\xi_{n}^{2})^{3}}\sum_{j,p,h,\delta,\varepsilon,q=1}^{n}\xi_{p}\xi_{\delta}a_{\varepsilon}^{\delta}a_{q}^{j}\partial_{x_{j}}(a_{h}^{p})c(dx_{\varepsilon})c(dx_{q})c(dx_{h})\nonumber\\
&-\frac{2}{(1+\xi_{n}^{2})^{3}}\sum_{j,p,h=1}^{n}\xi_{j}\xi_{p}\partial_{x_{j}}(a_{h}^{p})c(dx_{h})\nonumber
\end{align}
\begin{align}
&+\frac{1}{(1+\xi_{n}^{2})^{3}}\sum_{p,\kappa,o,e=1}^{n}\sum_{h=1}^{n-1}\xi_{p}\xi_{\kappa}a_{h}^{p}a_{o}^{\kappa}a_{e}^{n}c(dx_{o})c(dx_{e})\partial_{x_{n}}(c(dx_{h}))\nonumber\\
&-\frac{2}{(1+\xi_{n}^{2})^{3}}\sum_{p=1}^{n}\sum_{h=1}^{n-1}\xi_{n}\xi_{p}a_{h}^{p}\partial_{x_{n}}(c(dx_{h}))\nonumber\\
&-\frac{2}{(1+\xi_{n}^{2})^{3}}h'(0)\sum_{d, f, e, m, g=1}^{n}\xi_{d}\xi_{f}a_{e}^{d}a_{m}^{n}a_{g}^{f}c(dx_{e})c(dx_{m})c(dx_{g})\nonumber\\
&+\frac{4}{(1+\xi_{n}^{2})^{4}}h'(0)\sum_{\psi,\varphi=1}^{n}\xi_{n}\xi_{\psi}a_{\varphi}^{\psi}c(dx_{\varphi})\nonumber\\
&+\frac{1}{(1+\xi_{n}^{2})^{2}}\overline{c}(V).\nonumber
\end{align}

We note that $\int_{|\xi'|=1}{\{\xi_{i_1}\cdot\cdot\cdot\xi_{i_{2d+1}}}\}\sigma(\xi')=0,$ this gives
\begin{align}
i\int_{|\xi'|=1}\int^{+\infty}_{-\infty}{\rm tr}_{{\wedge^*T^*M}} [\widetilde{B}_1(x_0)\times
\sigma_{-4}({\widetilde{D}_{W}}^{-3})(x_0)]d\xi_n\sigma(\xi')dx'=0.\nonumber
\end{align}

Observing (2.32) and (3.42), we have
\begin{align}
&\sum_{l,\chi,\tau=1}^{n}\sum_{\gamma=1}^{n-1}{\rm tr}[c(dx_{l})\overline{c}(dx_{\gamma})\overline{c}(dx_{n})c(dx_{\tau})]=0;\\
&\sum_{l,\eta, \Phi, b, \Psi, c=1}^{n}\sum_{\nu=1}^{n-1}{\rm tr}[c(dx_{l})c(dx_{\Psi})c(dx_{\eta})\overline{c}(dx_{n})\overline{c}(dx_{\nu})c(dx_{c})]=0,
\end{align}
at this time,
\begin{align}
&i\int_{|\xi'|=1}\int^{+\infty}_{-\infty}{\rm tr}_{{\wedge^*T^*M}} [\widetilde{B}_2(x_0)\times
\sigma_{-4}({\widetilde{D}_{W}}^{-3})(x_0)]d\xi_n\sigma(\xi')dx'\\
&=i\int_{|\xi'|=1}\int^{+\infty}_{-\infty}-\frac{i}{2 \left(\xi _n-i\right)^5 \left(\xi _n+i\right)^3}h'(0)
\Big(\sum_{l, \eta, \Gamma, \Omega=1}^{n}\xi_{\Gamma}\xi_{\Omega} a_{l}^{n}a_{\eta}^{\Gamma}a_{\eta}^{n}a_{l}^{\Omega}{\rm tr}[\texttt{id}]\nonumber\\
&-\sum_{l, \Gamma, \Omega, \Lambda=1}^{n}\xi_{\Gamma}\xi_{\Omega} a_{l}^{n}a_{\Lambda}^{\Gamma}a_{l}^{n}a_{\Lambda}^{\Omega}{\rm tr}[\texttt{id}]+\sum_{l, \eta, \Gamma, \Omega=1}^{n}\xi_{\Gamma}\xi_{\Omega} a_{l}^{n}a_{l}^{\Gamma}a_{\eta}^{n}a_{\eta}^{\Omega}{\rm tr}[\texttt{id}]\Big)d\xi_n\sigma(\xi')dx'\nonumber\\
&+i\int_{|\xi'|=1}\int^{+\infty}_{-\infty}\frac{i}{2 \left(\xi _n-i\right)^5 \left(\xi _n+i\right)^3}h'(0)\Big(-\sum_{\chi=1}^{n}\sum_{\gamma=1}^{n-1}\xi_{\gamma}\xi_{\chi}a_{n}^{n}a_{\gamma}^{\chi}{\rm tr}[\texttt{id}]\nonumber\\
&+\sum_{\chi=1}^{n}\sum_{l=1}^{n-1}\xi_{l}\xi_{\chi}a_{l}^{n}a_{n}^{\chi}{\rm tr}[\texttt{id}]\Big)d\xi_n\sigma(\xi')dx'\nonumber
\end{align}
\begin{align}
&+i\int_{|\xi'|=1}\int^{+\infty}_{-\infty}\frac{5i}{2 \left(\xi _n-i\right)^5 \left(\xi _n+i\right)^3}h'(0)\sum_{l,\rho=1}^{n}\xi_{n}\xi_{\rho}a_{l}^{n}a_{l}^{\rho}{\rm tr}[\texttt{id}]d\xi_n\sigma(\xi')dx'\nonumber\\
&+i\int_{|\xi'|=1}\int^{+\infty}_{-\infty}-\frac{i}{\left(\xi _n-i\right)^5 \left(\xi _n+i\right)^3}\Big(\sum_{l,\alpha,\beta,\lambda=1}^{n}\xi_{\lambda}a_{l}^{n}a_{\alpha}^{\beta}a_{l}^{\lambda}g^{M}(dx_{\beta}, (\nabla^{L}_{e_{\alpha}}J)(\xi^{*})){\rm tr}[\texttt{id}]\nonumber\\
&-\sum_{l,\alpha,\beta,\lambda=1}^{n}\xi_{\lambda}a_{l}^{n}a_{\alpha}^{\beta}a_{\beta}^{\lambda}g^{M}(dx_{l}, (\nabla^{L}_{e_{\alpha}}J)(\xi^{*})){\rm tr}[\texttt{id}]\nonumber\\
&+\sum_{l,\alpha,\lambda,\omega=1}^{n}\xi_{\lambda}a_{l}^{n}a_{\alpha}^{l}a_{\omega}^{\lambda}g^{M}(dx_{\omega}, (\nabla^{L}_{e_{\alpha}}J)(\xi^{*})){\rm tr}[\texttt{id}]\Big)d\xi_n\sigma(\xi')dx'\nonumber\\
&+i\int_{|\xi'|=1}\int^{+\infty}_{-\infty}\frac{i}{8 \left(\xi _n-i\right)^5 \left(\xi _n+i\right)^3}h'(0)\Big(-\sum_{l, \Phi, b=1}^{n}\sum_{\nu=1}^{n-1}\xi_{\Phi}\xi_{b}a_{l}^{n}a_{\nu}^{n}a_{\nu}^{\Phi}a_{l}^{b}{\rm tr}[\texttt{id}]\nonumber\\
&+\sum_{l, \Phi, b=1}^{n}\sum_{\nu=1}^{n-1}\xi_{\Phi}\xi_{b}a_{l}^{n}a_{\nu}^{\nu}a_{n}^{\Phi}a_{l}^{b}{\rm tr}[\texttt{id}]+\sum_{\Phi, b, \Psi=1}^{n}\sum_{\nu=1}^{n-1}\xi_{\Phi}\xi_{b}a_{\nu}^{n}a_{\nu}^{n}a_{\Psi}^{\Phi}a_{\Psi}^{b}{\rm tr}[\texttt{id}]\nonumber\\
&-\sum_{\mu, \Phi, b=1}^{n}\sum_{\nu=1}^{n-1}\xi_{\Phi}\xi_{b}a_{\nu}^{n}a_{\nu}^{\mu}a_{n}^{\Phi}a_{\mu}^{b}{\rm tr}[\texttt{id}]+\sum_{\mu, \Phi, b=1}^{n}\sum_{\nu=1}^{n-1}\xi_{\Phi}\xi_{b}a_{\nu}^{n}a_{\nu}^{\mu}a_{\mu}^{\Phi}a_{n}^{b}{\rm tr}[\texttt{id}]\nonumber\\
&-\sum_{\Phi, b, \Psi=1}^{n}\sum_{\nu=1}^{n-1}\xi_{\Phi}\xi_{b}a_{n}^{n}a_{\nu}^{\nu}a_{\Psi}^{\Phi}a_{\Psi}^{b}{\rm tr}[\texttt{id}]+\sum_{\mu, \Phi, b=1}^{n}\sum_{\nu=1}^{n-1}\xi_{\Phi}\xi_{b}a_{n}^{n}a_{\nu}^{\mu}a_{\nu}^{\Phi}a_{\mu}^{b}{\rm tr}[\texttt{id}]\nonumber\\
&-\sum_{\mu, \Phi, b=1}^{n}\sum_{\nu=1}^{n-1}\xi_{\Phi}\xi_{b}a_{n}^{n}a_{\nu}^{\mu}a_{\mu}^{\Phi}a_{\nu}^{b}{\rm tr}[\texttt{id}]-\sum_{l, \Phi, b=1}^{n}\sum_{\nu=1}^{n-1}\xi_{\Phi}\xi_{b}a_{l}^{n}a_{\nu}^{l}a_{\nu}^{\Phi}a_{n}^{b}{\rm tr}[\texttt{id}]\nonumber\\
&+\sum_{l, \Phi, b=1}^{n}\sum_{\nu=1}^{n-1}\xi_{\Phi}\xi_{b}a_{l}^{n}a_{\nu}^{l}a_{n}^{\Phi}a_{\nu}^{b}{\rm tr}[\texttt{id}]+\sum_{l, \Phi, b=1}^{n}\sum_{\nu=1}^{n-1}\xi_{\Phi}\xi_{b}a_{l}^{n}a_{\nu}^{\nu}a_{l}^{\Phi}a_{n}^{b}{\rm tr}[\texttt{id}]\nonumber\\
&-\sum_{l, \Phi, b=1}^{n}\sum_{\nu=1}^{n-1}\xi_{\Phi}\xi_{b}a_{l}^{n}a_{\nu}^{n}a_{l}^{\Phi}a_{\nu}^{b}{\rm tr}[\texttt{id}]\Big)d\xi_n\sigma(\xi')dx'\nonumber\\
&+i\int_{|\xi'|=1}\int^{+\infty}_{-\infty}-\frac{i}{2 \left(\xi _n-i\right)^5 \left(\xi _n+i\right)^3}\Big(\sum_{l,j,p,\delta,q=1}^{n}\xi_{p}\xi_{\delta}a_{l}^{n}a_{q}^{\delta}a_{q}^{j}\partial_{x_{j}}(a_{l}^{p}){\rm tr}[\texttt{id}]\nonumber\\
&-\sum_{l,j,p,h,\delta=1}^{n}\xi_{p}\xi_{\delta}a_{l}^{n}a_{h}^{\delta}a_{l}^{j}\partial_{x_{j}}(a_{h}^{p}){\rm tr}[\texttt{id}]+\sum_{l,j,p,\delta,q=1}^{n}\xi_{p}\xi_{\delta}a_{l}^{n}a_{l}^{\delta}a_{q}^{j}\partial_{x_{j}}(a_{q}^{p}){\rm tr}[\texttt{id}]\Big)d\xi_n\sigma(\xi')dx'\nonumber\\
&+i\int_{|\xi'|=1}\int^{+\infty}_{-\infty}-\frac{i}{\left(\xi _n-i\right)^5 \left(\xi _n+i\right)^3}\sum_{l,j,p=1}^{n}\xi_{j}\xi_{p}a_{l}^{n}\partial_{x_{j}}(a_{l}^{p}){\rm tr}[\texttt{id}]d\xi_n\sigma(\xi')dx'\nonumber
\end{align}
\begin{align}
&+i\int_{|\xi'|=1}\int^{+\infty}_{-\infty}-\frac{i}{4 \left(\xi _n-i\right)^5 \left(\xi _n+i\right)^3}h'(0)\Big(\sum_{p,\kappa,o=1}^{n}\sum_{l=1}^{n-1}\xi_{p}\xi_{\kappa}a_{l}^{n}a_{l}^{p}a_{o}^{\kappa}a_{o}^{n}{\rm tr}[\texttt{id}]\nonumber\\
&-\sum_{l,p,\kappa=1}^{n}\sum_{h=1}^{n-1}\xi_{p}\xi_{\kappa}a_{l}^{n}a_{h}^{p}a_{h}^{\kappa}a_{l}^{n}{\rm tr}[\texttt{id}]+\sum_{l,p,\kappa=1}^{n}\sum_{h=1}^{n-1}\xi_{p}\xi_{\kappa}a_{l}^{n}a_{h}^{p}a_{l}^{\kappa}a_{h}^{n}{\rm tr}[\texttt{id}]\Big)d\xi_n\sigma(\xi')dx'\nonumber\\
&+i\int_{|\xi'|=1}\int^{+\infty}_{-\infty}-\frac{i}{2\left(\xi _n-i\right)^5 \left(\xi _n+i\right)^3}h'(0)\sum_{p=1}^{n}\sum_{l=1}^{n-1}\xi_{n}\xi_{p}a_{l}^{n}a_{l}^{p}{\rm tr}[\texttt{id}]d\xi_n\sigma(\xi')dx'\nonumber\\
&+i\int_{|\xi'|=1}\int^{+\infty}_{-\infty}\frac{i}{\left(\xi _n-i\right)^6 \left(\xi _n+i\right)^4}h'(0)\Big(\sum_{l, d, f, e=1}^{n}\xi_{d}\xi_{f}a_{l}^{n}a_{e}^{d}a_{e}^{n}a_{l}^{f}{\rm tr}[\texttt{id}]\nonumber\\
&-\sum_{l, d, f, e=1}^{n}\xi_{d}\xi_{f}a_{l}^{n}a_{e}^{d}a_{l}^{n}a_{e}^{f}{\rm tr}[\texttt{id}]+\sum_{l, d, f, m=1}^{n}\xi_{d}\xi_{f}a_{l}^{n}a_{l}^{d}a_{m}^{n}a_{m}^{f}{\rm tr}[\texttt{id}]\Big)d\xi_n\sigma(\xi')dx'\nonumber\\
&+i\int_{|\xi'|=1}\int^{+\infty}_{-\infty}\frac{2i}{\left(\xi _n-i\right)^6 \left(\xi _n+i\right)^4}h'(0)\sum_{l,\psi=1}^{n}\xi_{n}\xi_{\psi}a_{l}^{n}a_{l}^{\psi}{\rm tr}[\texttt{id}]d\xi_n\sigma(\xi')dx'.\nonumber
\end{align}

From $\int_{|\xi'|=1}\xi_i\xi_j=\frac{8\pi^{2}}{15}\delta_i^j,$ it follows that
\begin{align}
\overline{\Psi}_4
&=\Big(\sum_{l=1}^{n}g^{M}(J(dx_{l}), (\nabla^{L}_{e_{l}}J)e_{n}){\rm tr}[\texttt{id}]
-\sum_{l=1}^{n}g^{M}(J(dx_{n}), (\nabla^{L}_{e_{l}}J)e_{l}){\rm tr}[\texttt{id}]\nonumber\\
&+\sum_{l=1}^{n}g^{M}(J(dx_{l}), (\nabla^{L}_{e_{n}}J)e_{l}){\rm tr}[\texttt{id}]\Big)
\Omega_4(-\frac{\pi^{3}}{8})dx'\nonumber\\
&+\sum_{l,j=1}^{n}a_{l}^{j}\partial_{x_{j}}(a_{l}^{n}){\rm tr}[\texttt{id}]\Omega_4(\frac{\pi^{3}}{8})dx'\nonumber\\
&+\sum_{l,\beta=1}^{n}\sum_{i=1}^{n-1}(a_{l}^{n})^{2}(a_{\beta}^{i})^{2}{\rm tr}[\texttt{id}]\Omega_4h'(0)(\frac{\pi^{3}}{32})dx'\nonumber\\
&+\sum_{l=1}^{n}(a_{l}^{n})^{2}{\rm tr}[\texttt{id}]\Omega_4h'(0)(\frac{3\pi^{3}}{5})dx'\nonumber\\
&+\sum_{i=1}^{n-1}(a_{i}^{n})^{2}{\rm tr}[\texttt{id}]\Omega_4h'(0)(-\frac{5\pi^{3}}{64})dx'\nonumber\\
&+\sum_{i=1}^{n-1}a_{n}^{n}a_{i}^{i}{\rm tr}[\texttt{id}]\Omega_4h'(0)(-\frac{3\pi^{3}}{64})dx'.\nonumber
\end{align}

{\bf  case (c)}~$r=-2,l=-3,|\alpha|=j=k=0$\\

It is easily seen that
\begin{equation}
\overline{\Psi}_5=-i\int_{|\xi'|=1}\int^{+\infty}_{-\infty}{\rm tr}_{{\wedge^*T^*M}}[\pi^{+}_{\xi_{n}}\sigma_{-2}({\widetilde{D}_{W}}^{-1})
      \times\partial_{\xi_{n}}\sigma_{-3}({\widetilde{D}_{W}}^{-3})](x_0)d\xi_n\sigma(\xi')dx'.
\end{equation}

Likewise, we have
\begin{align}
&\pi^{+}_{\xi_{n}}\sigma_{-2}({\widetilde{D}_{W}}^{-1})(x_0)|_{|\xi'|=1}=\pi^{+}_{\xi_{n}}\Big(\frac{c[J(\xi)]\widetilde{b}_0^2(x_0)c[J(\xi)]}{(1+\xi_n^2)^2}\Big)-h'(0)\pi^{+}_{\xi_{n}}\Big(\frac{c[J(\xi)]}{(1+\xi_n^2)^3}c[J(dx_n)]c[J(\xi)]\Big)\\
&+\pi^{+}_{\xi_{n}}\Big(\frac{c[J(\xi)]}{(1+\xi_n^2)^2}\Big[\sum_{j,p,h=1}^{n}\xi_p\partial_{x_j}(a_{h}^{p})c[J(dx_j)]c(dx_h)+\sum_{p=1}^{n}\sum_{h=1}^{n-1}\xi_pa_{h}^{p}c[J(dx_n)]\partial_{x_n}(c(dx_h))\Big]\Big)\nonumber\\
&+\pi^{+}_{\xi_{n}}\Big(\frac{c[J(\xi)]\widetilde{b}_0^1(x_0)c[J(\xi)]}{(1+\xi_n^2)^2}\Big)+\pi^{+}_{\xi_{n}}\Big(\frac{c[J(\xi)]\overline{c}(V)c[J(\xi)]}{(1+\xi_n^2)^2}\Big).\nonumber
\end{align}
It is clear that
\begin{align}
&\pi^{+}_{\xi_{n}}\Big(\frac{c[J(\xi)]\widetilde{b}_0^1(x_0)c[J(\xi)]}{(1+\xi_n^2)^2}\Big)\\
&=-\frac{i\xi_n}{16(\xi_n-i)^2}h'(0)\sum_{l,\gamma,\eta=1}^{n}\sum_{\nu=1}^{n-1}a_{l}^{n}a_{\gamma}^{n}a_{\nu}^{\eta}c(dx_l)c(dx_\eta)\overline{c}(dx_n)\overline{c}(dx_\nu)c(dx_{\gamma})\nonumber\\
&-\frac{i}{16(\xi_n-i)^2}h'(0)\sum_{l,\gamma,\eta=1}^{n}\sum_{q,\nu=1}^{n-1}\xi_{q}a_{l}^{q}a_{\gamma}^{n}a_{\nu}^{\eta}c(dx_l)c(dx_\eta)\overline{c}(dx_n)\overline{c}(dx_\nu)c(dx_{\gamma})\nonumber\\
&-\frac{i}{16(\xi_n-i)^2}h'(0)\sum_{l,\gamma,\eta=1}^{n}\sum_{\alpha,\nu=1}^{n-1}\xi_{\alpha}a_{l}^{n}a_{\gamma}^{\alpha}a_{\nu}^{\eta}c(dx_l)c(dx_\eta)\overline{c}(dx_n)\overline{c}(dx_\nu)c(dx_{\gamma})\nonumber\\
&-\frac{i\xi_n+2}{16(\xi_n-i)^2}h'(0)\sum_{l,\gamma,\eta=1}^{n}\sum_{q,\alpha,\nu=1}^{n-1}\xi_{q}\xi_{\alpha}a_{l}^{q}a_{\gamma}^{\alpha}a_{\nu}^{\eta}c(dx_l)c(dx_\eta)\overline{c}(dx_n)\overline{c}(dx_\nu)c(dx_{\gamma}),\nonumber
\end{align}
\begin{align}
&\pi^{+}_{\xi_{n}}\Big(\frac{c[J(\xi)]\overline{c}(V)c[J(\xi)]}{(1+\xi_n^2)^2}\Big)\\
&=-\frac{i\xi_n}{4(\xi_n-i)^2}\sum_{l,\gamma=1}^{n}a_{l}^{n}a_{\gamma}^{n}c(dx_l)\overline{c}(V)c(dx_{\gamma})\nonumber\\
&-\frac{i}{4(\xi_n-i)^2}\sum_{l,\gamma=1}^{n}\sum_{q=1}^{n-1}\xi_{q}a_{l}^{q}a_{\gamma}^{n}c(dx_l)\overline{c}(V)c(dx_{\gamma})\nonumber
\end{align}
\begin{align}
&-\frac{i}{4(\xi_n-i)^2}\sum_{l,\gamma=1}^{n}\sum_{\alpha=1}^{n-1}\xi_{\alpha}a_{l}^{n}a_{\gamma}^{\alpha}c(dx_l)\overline{c}(V)c(dx_{\gamma})\nonumber\\
&-\frac{i\xi_n+2}{4(\xi_n-i)^2}\sum_{l,\gamma=1}^{n}\sum_{q,\alpha=1}^{n-1}\xi_{q}\xi_{\alpha}a_{l}^{q}a_{\gamma}^{\alpha}c(dx_l)\overline{c}(V)c(dx_{\gamma}).\nonumber
\end{align}

A simple calculation shows that
\begin{align}
\partial_{\xi_n}\left(\frac{ic[J(\xi)]}{|\xi|^{4}}\right)(x_0)|_{|\xi'|=1}
&=i\sum^{n}_{\beta=1}\sum^{n-1}_{i=1}\xi_{i}a_{\beta}^ic(dx_{\beta})\partial_{\xi_n}\left(\frac{1}{(1+\xi_{n}^{2})^{2}}\right)+i\sum^{n}_{\beta=1}a_{\beta}^nc(dx_{\beta})\partial_{\xi_n}\left(\frac{\xi_{n}}{(1+\xi_{n}^{2})^{2}}\right)\\
&=-\frac{4 i \xi_n}{\left(1+\xi_n^2\right)^3}\sum^{n}_{\beta=1}\sum^{n-1}_{i=1}\xi_{i}a_{\beta}^ic(dx_{\beta})+\frac{i(1-3 \xi _n^2)}{\left(1+\xi_n^2\right)^3}\sum^{n}_{\beta=1}a_{\beta}^nc(dx_{\beta}).\nonumber
\end{align}

Accordingly, we have
\begin{align}
&{\rm tr} [\pi^{+}_{\xi_{n}}\Big(\frac{c[J(\xi)]\widetilde{b}_0^1(x_0)c[J(\xi)]}{(1+\xi_n^2)^2}\Big) \times \partial_{\xi_n}\sigma_{-3}({\widetilde{D}_{W}}^{-3})(x_0)]|_{|\xi'|=1}\\
&=-\frac{\xi _n^2}{4 \left(\xi _n-i\right)^5 \left(\xi _n+i\right)^3}h'(0)\sum_{l,\gamma,\eta,\beta=1}^{n}\sum_{\nu,i=1}^{n-1}{\rm tr}[\xi_{i}a_{l}^{n}a_{\gamma}^{n}a_{\nu}^{\eta}a_{\beta}^{i}c(dx_{l})c(dx_{\eta})\overline{c}(dx_{n})\overline{c}(dx_{\nu})c(dx_{\gamma})c(dx_{\beta})]\nonumber\\
&-\frac{\xi _n \left(3 \xi _n^2-1\right)}{16 \left(\xi _n-i\right)^5 \left(\xi _n+i\right)^3}h'(0)\sum_{l,\gamma,\eta,\beta=1}^{n}\sum_{\nu=1}^{n-1}{\rm tr}[a_{l}^{n}a_{\gamma}^{n}a_{\nu}^{\eta}a_{\beta}^{n}c(dx_{l})c(dx_{\eta})\overline{c}(dx_{n})\overline{c}(dx_{\nu})c(dx_{\gamma})c(dx_{\beta})]\nonumber\\
&-\frac{\xi _n}{4 \left(\xi _n-i\right)^5 \left(\xi _n+i\right)^3}h'(0)\sum_{l,\gamma,\eta,\beta=1}^{n}\sum_{q,\nu,i=1}^{n-1}{\rm tr}[\xi_{q}\xi_{i}a_{l}^{q}a_{\gamma}^{n}a_{\nu}^{\eta}a_{\beta}^{i}c(dx_{l})c(dx_{\eta})\overline{c}(dx_{n})\overline{c}(dx_{\nu})c(dx_{\gamma})c(dx_{\beta})]\nonumber\\
&-\frac{3 \xi _n^2-1}{16 \left(\xi _n-i\right)^5 \left(\xi _n+i\right)^3}h'(0)\sum_{l,\gamma,\eta,\beta=1}^{n}\sum_{q,\nu=1}^{n-1}{\rm tr}[\xi_{q}a_{l}^{q}a_{\gamma}^{n}a_{\nu}^{\eta}a_{\beta}^{n}c(dx_{l})c(dx_{\eta})\overline{c}(dx_{n})\overline{c}(dx_{\nu})c(dx_{\gamma})c(dx_{\beta})]\nonumber\\
&-\frac{\xi _n}{4 \left(\xi _n-i\right)^5 \left(\xi _n+i\right)^3}h'(0)\sum_{l,\gamma,\eta,\beta=1}^{n}\sum_{\alpha,\nu,i=1}^{n-1}{\rm tr}[\xi_{\alpha}\xi_{i}a_{l}^{n}a_{\gamma}^{\alpha}a_{\nu}^{\eta}a_{\beta}^{i}c(dx_{l})c(dx_{\eta})\overline{c}(dx_{n})\overline{c}(dx_{\nu})c(dx_{\gamma})c(dx_{\beta})]\nonumber\\
&-\frac{3 \xi _n^2-1}{16 \left(\xi _n-i\right)^5 \left(\xi _n+i\right)^3}h'(0)\sum_{l,\gamma,\eta,\beta=1}^{n}\sum_{\alpha,\nu=1}^{n-1}{\rm tr}[\xi_{\alpha}a_{l}^{n}a_{\gamma}^{\alpha}a_{\nu}^{\eta}a_{\beta}^{n}c(dx_{l})c(dx_{\eta})\overline{c}(dx_{n})\overline{c}(dx_{\nu})c(dx_{\gamma})c(dx_{\beta})]\nonumber\\
&-\frac{\xi _n \left(\xi _n-2 i\right)}{4 \left(\xi _n-i\right)^5 \left(\xi _n+i\right)^3}h'(0)\sum_{l,\gamma,\eta,\beta=1}^{n}\sum_{q,\alpha,\nu,i=1}^{n-1}{\rm tr}[\xi_{q}\xi_{\alpha}\xi_{i}a_{l}^{q}a_{\gamma}^{\alpha}a_{\nu}^{\eta}a_{\beta}^{i}c(dx_{l})c(dx_{\eta})\overline{c}(dx_{n})\overline{c}(dx_{\nu})c(dx_{\gamma})c(dx_{\beta})]\nonumber
\end{align}
\begin{align}
&-\frac{\left(\xi _n-2 i\right) \left(3 \xi _n^2-1\right)}{16 \left(\xi _n-i\right)^5 \left(\xi _n+i\right)^3}h'(0)\sum_{l,\gamma,\eta,\beta=1}^{n}\sum_{q,\alpha,\nu=1}^{n-1}{\rm tr}[\xi_{q}\xi_{\alpha}a_{l}^{q}a_{\gamma}^{\alpha}a_{\nu}^{\eta}a_{\beta}^{n}c(dx_{l})c(dx_{\eta})\overline{c}(dx_{n})\overline{c}(dx_{\nu})c(dx_{\gamma})c(dx_{\beta})].\nonumber
\end{align}
By $\sum_{l,\gamma,\eta,\beta=1}^{n}\sum_{\nu=1}^{n-1}{\rm tr}[c(dx_{l})c(dx_{\eta})\overline{c}(dx_{n})\overline{c}(dx_{\nu})c(dx_{\gamma})c(dx_{\beta})]=0,$ we see that
\begin{align}
&-i\int_{|\xi'|=1}\int^{+\infty}_{-\infty}{\rm tr}_{{\wedge^*T^*M}}[\pi^{+}_{\xi_{n}}\Big(\frac{c[J(\xi)]\widetilde{b}_0^1(x_0)c[J(\xi)]}{(1+\xi_n^2)^2}\Big) \times \partial_{\xi_n}\sigma_{-3}({\widetilde{D}_{W}}^{-3})(x_0)]d\xi_n\sigma(\xi')dx'=0.\nonumber
\end{align}

A trivial verification shows that
\begin{align}
&{\rm tr} [\pi^{+}_{\xi_{n}}\Big(\frac{c[J(\xi)]\overline{c}(V)c[J(\xi)]}{(1+\xi_n^2)^2}\Big) \times \partial_{\xi_n}\sigma_{-3}({\widetilde{D}_{W}}^{-3})(x_0)]|_{|\xi'|=1}\\
&=-\frac{\xi _n^2}{\left(\xi _n-i\right)^5 \left(\xi _n+i\right)^3}h'(0)\sum_{l,\gamma,\beta=1}^{n}\sum_{i=1}^{n-1}{\rm tr}[\xi_{i}a_{l}^{n}a_{\gamma}^{n}a_{\beta}^{i}c(dx_{l})\overline{c}(V)c(dx_{\gamma})c(dx_{\beta})]\nonumber\\
&-\frac{\xi _n \left(3 \xi _n^2-1\right)}{4 \left(\xi _n-i\right)^5 \left(\xi _n+i\right)^3}h'(0)\sum_{l,\gamma,\beta=1}^{n}{\rm tr}[a_{l}^{n}a_{\gamma}^{n}a_{\beta}^{n}c(dx_{l})\overline{c}(V)c(dx_{\gamma})c(dx_{\beta})]\nonumber\\
&-\frac{\xi _n}{\left(\xi _n-i\right)^5 \left(\xi _n+i\right)^3}h'(0)\sum_{l,\gamma,\beta=1}^{n}\sum_{q,i=1}^{n-1}{\rm tr}[\xi_{q}\xi_{i}a_{l}^{q}a_{\gamma}^{n}a_{\beta}^{i}c(dx_{l})\overline{c}(V)c(dx_{\gamma})c(dx_{\beta})]\nonumber\\
&-\frac{3 \xi _n^2-1}{4 \left(\xi _n-i\right)^5 \left(\xi _n+i\right)^3}h'(0)\sum_{l,\gamma,\beta=1}^{n}\sum_{q=1}^{n-1}{\rm tr}[\xi_{q}a_{l}^{q}a_{\gamma}^{n}a_{\beta}^{n}c(dx_{l})\overline{c}(V)c(dx_{\gamma})c(dx_{\beta})]\nonumber\\
&-\frac{\xi _n}{\left(\xi _n-i\right)^5 \left(\xi _n+i\right)^3}h'(0)\sum_{l,\gamma,\beta=1}^{n}\sum_{\alpha,i=1}^{n-1}{\rm tr}[\xi_{\alpha}\xi_{i}a_{l}^{n}a_{\gamma}^{\alpha}a_{\beta}^{i}c(dx_{l})\overline{c}(V)c(dx_{\gamma})c(dx_{\beta})]\nonumber\\
&-\frac{3 \xi _n^2-1}{4 \left(\xi _n-i\right)^5 \left(\xi _n+i\right)^3}h'(0)\sum_{l,\gamma,\beta=1}^{n}\sum_{\alpha=1}^{n-1}{\rm tr}[\xi_{\alpha}a_{l}^{n}a_{\gamma}^{\alpha}a_{\beta}^{n}c(dx_{l})\overline{c}(V)c(dx_{\gamma})c(dx_{\beta})]\nonumber\\
&-\frac{\xi _n \left(\xi _n-2 i\right)}{\left(\xi _n-i\right)^5 \left(\xi _n+i\right)^3}h'(0)\sum_{l,\gamma,\beta=1}^{n}\sum_{q,\alpha,i=1}^{n-1}{\rm tr}[\xi_{q}\xi_{\alpha}\xi_{i}a_{l}^{q}a_{\gamma}^{\alpha}a_{\beta}^{i}c(dx_{l})\overline{c}(V)c(dx_{\gamma})c(dx_{\beta})]\nonumber\\
&-\frac{\left(\xi _n-2 i\right) \left(3 \xi _n^2-1\right)}{4 \left(\xi _n-i\right)^5 \left(\xi _n+i\right)^3}h'(0)\sum_{l,\gamma,\beta=1}^{n}\sum_{q,\alpha=1}^{n-1}{\rm tr}[\xi_{q}\xi_{\alpha}a_{l}^{q}a_{\gamma}^{\alpha}a_{\beta}^{n}c(dx_{l})\overline{c}(V)c(dx_{\gamma})c(dx_{\beta})].\nonumber
\end{align}
Using $\int_{|\xi'|=1}{\{\xi_{i_1}\cdot\cdot\cdot\xi_{i_{2d+1}}}\}\sigma(\xi')=0,$ we conclude that
\begin{align}
&-i\int_{|\xi'|=1}\int^{+\infty}_{-\infty}{\rm tr}_{{\wedge^*T^*M}}[\pi^{+}_{\xi_{n}}\Big(\frac{c[J(\xi)]\overline{c}(V)c[J(\xi)]}{(1+\xi_n^2)^2}\Big) \times \partial_{\xi_n}\sigma_{-3}({\widetilde{D}_{W}}^{-3})(x_0)]d\xi_n\sigma(\xi')dx'
\end{align}
\begin{align}
&=i\int_{|\xi'|=1}\int^{+\infty}_{-\infty}\frac{\xi _n \left(3 \xi _n^2-1\right)h'(0)}{4 \left(\xi _n-i\right)^5 \left(\xi _n+i\right)^3}\sum_{l,\gamma,\beta=1}^{n}{\rm tr}[a_{l}^{n}a_{\gamma}^{n}a_{\beta}^{n}c(dx_{l})\overline{c}(V)c(dx_{\gamma})c(dx_{\beta})]d\xi_n\sigma(\xi')dx'\nonumber\\
&+i\int_{|\xi'|=1}\int^{+\infty}_{-\infty}\frac{\xi _n h'(0)}{\left(\xi _n-i\right)^5 \left(\xi _n+i\right)^3}\sum_{l,\gamma,\beta=1}^{n}\sum_{q,i=1}^{n-1}{\rm tr}[\xi_{q}\xi_{i}a_{l}^{q}a_{\gamma}^{n}a_{\beta}^{i}c(dx_{l})\overline{c}(V)c(dx_{\gamma})c(dx_{\beta})]d\xi_n\sigma(\xi')dx'\nonumber\\
&+i\int_{|\xi'|=1}\int^{+\infty}_{-\infty}\frac{\xi _n h'(0)}{\left(\xi _n-i\right)^5 \left(\xi _n+i\right)^3}\sum_{l,\gamma,\beta=1}^{n}\sum_{\alpha,i=1}^{n-1}{\rm tr}[\xi_{\alpha}\xi_{i}a_{l}^{n}a_{\gamma}^{\alpha}a_{\beta}^{i}c(dx_{l})\overline{c}(V)c(dx_{\gamma})c(dx_{\beta})]d\xi_n\sigma(\xi')dx'\nonumber\\
&+i\int_{|\xi'|=1}\int^{+\infty}_{-\infty}\frac{\left(\xi _n-2 i\right) \left(3 \xi _n^2-1\right) h'(0)}{4 \left(\xi _n-i\right)^5 \left(\xi _n+i\right)^3}\sum_{l,\gamma,\beta=1}^{n}\sum_{q,\alpha=1}^{n-1}{\rm tr}[\xi_{q}\xi_{\alpha}a_{l}^{q}a_{\gamma}^{\alpha}a_{\beta}^{n}c(dx_{l})\overline{c}(V)c(dx_{\gamma})c(dx_{\beta})]d\xi_n\sigma(\xi')dx',\nonumber
\end{align}
of course,
\begin{align}
-i\int_{|\xi'|=1}\int^{+\infty}_{-\infty}{\rm tr}_{{\wedge^*T^*M}}[\pi^{+}_{\xi_{n}}\Big(\frac{c[J(\xi)]\overline{c}(V)c[J(\xi)]}{(1+\xi_n^2)^2}\Big) \times \partial_{\xi_n}\sigma_{-3}({\widetilde{D}_{W}}^{-3})(x_0)]d\xi_n\sigma(\xi')dx'=0.
\end{align}

On account of the above result,
\begin{align}
\overline{\Psi}_5&=-i\int_{|\xi'|=1}\int^{+\infty}_{-\infty}{\rm tr}_{{\wedge^*T^*M}}[\pi^{+}_{\xi_{n}}\Big(\frac{c[J(\xi)]\widetilde{b}_0^1(x_0)c[J(\xi)]}{(1+\xi_n^2)^2}\Big) \times \partial_{\xi_n}\sigma_{-3}({\widetilde{D}_{W}}^{-3})(x_0)]d\xi_n\sigma(\xi')dx'\\
&-i\int_{|\xi'|=1}\int^{+\infty}_{-\infty}{\rm tr}_{{\wedge^*T^*M}}[\pi^{+}_{\xi_{n}}\Big(\frac{c[J(\xi)]\overline{c}(V)c[J(\xi)]}{(1+\xi_n^2)^2}\Big) \times \partial_{\xi_n}\sigma_{-3}({\widetilde{D}_{W}}^{-3})(x_0)]d\xi_n\sigma(\xi')dx'\nonumber\\
&-i\int_{|\xi'|=1}\int^{+\infty}_{-\infty}{\rm tr}_{{\wedge^*T^*M}}[\pi^{+}_{\xi_{n}}\Big(\frac{c[J(\xi)]\widetilde{b}_0^2(x_0)c[J(\xi)]}{(1+\xi_n^2)^2}\Big)-h'(0)\pi^{+}_{\xi_{n}}\Big(\frac{c[J(\xi)]}{(1+\xi_n^2)^3}c[J(dx_n)]c[J(\xi)]\Big)\nonumber\\
&+\pi^{+}_{\xi_{n}}\Big(\frac{c[J(\xi)]}{(1+\xi_n^2)^2}\Big[\sum_{j,p,h=1}^{n}\xi_p\partial_{x_j}(a_{h}^{p})c[J(dx_j)]c(dx_h)+\sum_{p=1}^{n}\sum_{h=1}^{n-1}\xi_pa_{h}^{p}c[J(dx_n)]\partial_{x_n}(c(dx_h))\Big]\Big)\nonumber\\
&\times \partial_{\xi_n}\sigma_{-3}({\widetilde{D}_{W}}^{-3})(x_0)]d\xi_n\sigma(\xi')dx'\nonumber\\
&=\sum_{l,\beta=1}^{n}\sum_{i=1}^{n-1}a_{l}^{i}a_{\beta}^{i}a_{l}^{n}a_{\beta}^{n}{\rm tr}[\texttt{id}]\Omega_4h'(0)(-\frac{7 \pi ^3}{80} )dx'
+\sum_{l=1}^{n}\sum_{\nu,i=1}^{n-1}a_{\nu}^{i}a_{l}^{i}a_{\nu}^{n}a_{l}^{n}{\rm tr}[\texttt{id}]\Omega_4h'(0)(\frac{\pi ^3}{16})dx'\nonumber\\
&+\sum_{l,j,\beta=1}^{n}(a_{\beta}^{n})^2a_{l}^{j}\partial_{x_j}(a_{l}^{n}){\rm tr}[\texttt{id}]\Omega_4(-\frac{\pi }{64})dx'
+\sum_{l,j,\beta=1}^{n}\sum_{i=1}^{n-1}(a_{\beta}^{i})^2a_{l}^{j}\partial_{x_j}(a_{l}^{n}){\rm tr}[\texttt{id}]\Omega_4(-\frac{\pi ^3}{24} )dx'\nonumber\\
&+\sum_{l=1}^{n}\sum_{i=1}^{n-1}(a_{l}^{n})^2a_{i}^{i}a_{n}^{n}{\rm tr}[\texttt{id}]\Omega_4h'(0)(\frac{\pi }{256})dx'
+\sum_{l=1}^{n}\sum_{\nu,i=1}^{n-1}(a_{l}^{i})^2a_{\nu}^{\nu}a_{n}^{n}{\rm tr}[\texttt{id}]\Omega_4h'(0)(\frac{\pi ^3}{32})dx'\nonumber
\end{align}
\begin{align}
&+\sum_{l,j=1}^{n}a_{l}^{j}\partial_{x_j}(a_{l}^{n}){\rm tr}[\texttt{id}]\Omega_4(-\frac{\pi^{3} }{12})dx'
+\sum_{l,\beta=1}^{n}(a_{\beta}^{n})^2(a_{l}^{n})^2{\rm tr}[\texttt{id}]\Omega_4h'(0)(\frac{\pi }{256})dx'\nonumber\\
&+\sum_{l,\beta=1}^{n}\sum_{i=1}^{n-1}(a_{l}^{n})^2(a_{\beta}^{i})^2{\rm tr}[\texttt{id}]\Omega_4h'(0)(\frac{49 \pi ^3}{480})dx'
+\sum_{l=1}^{n}\sum_{i=1}^{n-1}(a_{i}^{n})^2(a_{l}^{n})^2{\rm tr}[\texttt{id}]\Omega_4h'(0)(-\frac{3 \pi }{256})dx'\nonumber\\
&+\sum_{l=1}^{n}\sum_{\nu,i=1}^{n-1}(a_{\nu}^{n})^2(a_{l}^{i})^2{\rm tr}[\texttt{id}]\Omega_4h'(0)(-\frac{5 \pi ^3}{96} )dx'
+\sum_{l=1}^{n}\sum_{\nu,i=1}^{n-1}(a_{l}^{n})^2(a_{\nu}^{i})^2{\rm tr}[\texttt{id}]\Omega_4h'(0)(-\frac{ \pi ^3}{24} )dx'\nonumber\\
&+\sum_{\nu,i=1}^{n-1}a_{i}^{i}a_{\nu}^{\nu}{\rm tr}[\texttt{id}]\Omega_4h'(0)(-\frac{\pi ^3}{48})dx'.\nonumber
\end{align}

In summary,
\begin{align}
\overline{\Psi}
&=\overline{\Psi}_1+\overline{\Psi}_2+\overline{\Psi}_3+\overline{\Psi}_4+\overline{\Psi}_5\\
&=\Big(\sum_{l=1}^{n}g^{M}(J(dx_{l}), (\nabla^{L}_{e_{l}}J)e_{n}){\rm tr}[\texttt{id}]
-\sum_{l=1}^{n}g^{M}(J(dx_{n}), (\nabla^{L}_{e_{l}}J)e_{l}){\rm tr}[\texttt{id}]\nonumber\\
&+\sum_{l=1}^{n}g^{M}(J(dx_{l}), (\nabla^{L}_{e_{n}}J)e_{l}){\rm tr}[\texttt{id}]\Big)
\Omega_4(-\frac{\pi ^3}{8})dx'
+\sum_{l,j=1}^{n}a_{l}^{j}\partial_{x_{j}}(a_{l}^{n}){\rm tr}[\texttt{id}]\Omega_4(\frac{1}{48} \pi  \left(10 \pi ^2-3\right))dx'\nonumber\\
&+\sum_{l,j,\beta=1}^{n}(a_{\beta}^{n})^2a_{l}^{j}\partial_{x_j}(a_{l}^{n}){\rm tr}[\texttt{id}]\Omega_4(-\frac{\pi }{64})dx'
+\sum_{l,j,\beta=1}^{n}\sum_{i=1}^{n-1}(a_{\beta}^{i})^2a_{l}^{j}\partial_{x_j}(a_{l}^{n}){\rm tr}[\texttt{id}]\Omega_4(-\frac{\pi ^3}{24} )dx'\nonumber\\
&+\sum_{l,\beta=1}^{n}\sum_{i=1}^{n-1}a_{l}^{i}a_{\beta}^{i}a_{l}^{n}a_{\beta}^{n}{\rm tr}[\texttt{id}]\Omega_4h'(0)(-\frac{7 \pi ^3}{80} )dx'
+\sum_{l=1}^{n}\sum_{\nu,i=1}^{n-1}a_{\nu}^{i}a_{l}^{i}a_{\nu}^{n}a_{l}^{n}{\rm tr}[\texttt{id}]\Omega_4h'(0)(\frac{\pi ^3}{16})dx'\nonumber\\
&+\sum_{l=1}^{n}\sum_{i=1}^{n-1}(a_{l}^{n})^2a_{i}^{i}a_{n}^{n}{\rm tr}[\texttt{id}]\Omega_4h'(0)(\frac{\pi }{256})dx'
+\sum_{l=1}^{n}\sum_{\nu,i=1}^{n-1}(a_{l}^{i})^2a_{\nu}^{\nu}a_{n}^{n}{\rm tr}[\texttt{id}]\Omega_4h'(0)(\frac{\pi ^3}{32})dx'\nonumber\\
&+\sum_{l,\beta=1}^{n}(a_{\beta}^{n})^2(a_{l}^{n})^2{\rm tr}[\texttt{id}]\Omega_4h'(0)(\frac{\pi }{256})dx'
+\sum_{l,\beta=1}^{n}\sum_{i=1}^{n-1}(a_{l}^{n})^2(a_{\beta}^{i})^2{\rm tr}[\texttt{id}]\Omega_4h'(0)(\frac{2 \pi ^3}{15})dx'\nonumber\\
&+\sum_{l=1}^{n}\sum_{i=1}^{n-1}(a_{i}^{n})^2(a_{l}^{n})^2{\rm tr}[\texttt{id}]\Omega_4h'(0)(-\frac{3 \pi}{256})dx'
+\sum_{l=1}^{n}\sum_{\nu,i=1}^{n-1}(a_{\nu}^{n})^2(a_{l}^{i})^2{\rm tr}[\texttt{id}]\Omega_4h'(0)(-\frac{5 \pi ^3}{96})dx'\nonumber\\
&+\sum_{l=1}^{n}\sum_{\nu,i=1}^{n-1}(a_{l}^{n})^2(a_{\nu}^{i})^2{\rm tr}[\texttt{id}]\Omega_4h'(0)(-\frac{\pi ^3}{24} )dx'
+\sum_{i=1}^{n-1}a_{n}^{n}a_{i}^{i}{\rm tr}[\texttt{id}]\Omega_4h'(0)(-\frac{3 \pi^3}{64})dx'\nonumber
\end{align}
\begin{align}
&+\sum_{\nu,i=1}^{n-1}a_{i}^{i}a_{\nu}^{\nu}{\rm tr}[\texttt{id}]\Omega_4h'(0)(-\frac{\pi ^3}{48})dx'
+\sum_{l=1}^{n}(a_{l}^{n})^{2}{\rm tr}[\texttt{id}]\Omega_4h'(0)(\frac{1}{640} \pi  \left(384 \pi ^2+15\right))dx'\nonumber\\
&+\sum_{i=1}^{n-1}(a_{i}^{n})^{2}{\rm tr}[\texttt{id}]\Omega_4h'(0)(-\frac{5 \pi^3}{64})dx'
+\sum_{l=1}^{n}\sum_{i=1}^{n-1}(a_{l}^{i})^2{\rm tr}[\texttt{id}]\Omega_4h'(0)(\frac{7 \pi ^3}{240})dx'.\nonumber
\end{align}

Combine (4.3) with (4.32), we obtain immediately the following theorem:
\begin{thm}
Let $M$ be a $6$-dimensional almost product Riemannian spin manifold with the boundary $\partial M$ and the metric $g^M$ as above, $\widetilde{D}_{W}$ be $J$-Witten deformation on $M$, then
\begin{align}
&\widetilde{{\rm Wres}}[\pi^+{\widetilde{D}_{W}}^{-1}\circ\pi^+{\widetilde{D}_{W}}^{-3}]\\
&=32\pi^{3}\int_{M}\Big(-2\sum_{\nu,j=1}^{6}g^{M}(\nabla_{e_{j}}^{L}(J)e_{\nu}, (\nabla^{L}_{e_{\nu}}J)e_{j})-2\sum_{\nu,j=1}^{6}g^{M}(J(e_{\nu}), (\nabla^{L}_{e_{j}}(\nabla^{L}_{e_{\nu}}(J)))e_{j}-(\nabla^{L}_{\nabla^{L}_{e_{j}}e_{\nu}}(J))e_{j})\nonumber\\
&-\sum_{\alpha,\nu,j=1}^{6}g^{M}(J(e_{\alpha}), (\nabla^{L}_{e_{\nu}}J)e_{j})g^{M}((\nabla^{L}_{e_{\alpha}}J)e_{j}, J(e_{\nu}))-\sum_{\alpha,\nu,j=1}^{6}g^{M}(J(e_{\alpha}), (\nabla^{L}_{e_{\alpha}}J)e_{j})g^{M}(J(e_{\nu}), (\nabla^{L}_{e_{\nu}}J)e_{j})\nonumber\\
&+\sum_{\nu,j=1}^{6}g^{M}((\nabla^{L}_{e_{\nu}}J)e_{j}, (\nabla^{L}_{e_{\nu}}J)e_{j}))-\frac{5}{3}s-4|V|^2\Big)d{\rm Vol_{M} }\nonumber\\
&+\int_{\partial M}\Big[
\frac{4\pi}{3}\left(10 \pi ^2-3\right)
\Big(\sum_{l=1}^{6}\langle\nabla_{J(e_{l})}^{L}(Je_{6}), e_{l}\rangle-\sum_{l=1}^{6}g^{M}\left(J(\frac{\partial}{\partial{x_{6}}}), \frac{\partial}{\partial{x_{6}}}\right)\langle\nabla_{J(e_{l})}^{L}(\frac{\partial}{\partial{x_{6}}}), e_{l}\rangle\Big)\nonumber\\
&-\pi\sum_{\beta=1}^{6}\langle J(e_{\beta}), e_{6}\rangle^{2}\Big(\sum_{l=1}^{6}\langle\nabla_{J(e_{l})}^{L}(Je_{6}), e_{l}\rangle-\sum_{l=1}^{6}g^{M}\left(J(\frac{\partial}{\partial{x_{6}}}), \frac{\partial}{\partial{x_{6}}}\right)\langle\nabla_{J(e_{l})}^{L}(\frac{\partial}{\partial{x_{6}}}), e_{l}\rangle\Big)\nonumber\\
&-\frac{8\pi ^3}{3}\sum_{\beta=1}^{6}\sum_{i=1}^{5}\langle J(e_{\beta}), e_{i}\rangle^{2}\Big(\sum_{l=1}^{6}\langle\nabla_{J(e_{l})}^{L}(Je_{6}), e_{l}\rangle-\sum_{l=1}^{6}g^{M}\left(J(\frac{\partial}{\partial{x_{6}}}), \frac{\partial}{\partial{x_{6}}}\right) \langle\nabla_{J(e_{l})}^{L}(\frac{\partial}{\partial{x_{6}}}), e_{l}\rangle\Big)\nonumber\\
&-8\pi^3\Big(\sum_{l=1}^{6}g^{M}(J(e_{l}), (\nabla^{L}_{e_{l}}J)e_{6})
-\sum_{l=1}^{6}g^{M}(J(\frac{\partial}{\partial{x_{6}}}), (\nabla^{L}_{e_{l}}J)e_{l})+\sum_{l=1}^{6}g^{M}(J(e_{l}), (\nabla^{L}_{e_{6}}J)e_{l})\Big)\nonumber\\
&+\frac{\pi }{4}h'(0)\sum_{l,\beta=1}^{6}\langle J(e_{\beta}), e_{6}\rangle^{2} \langle J(e_{l}), e_{6}\rangle^{2}
+\frac{128 \pi ^3}{15}h'(0)\sum_{l,\beta=1}^{6}\sum_{i=1}^{5}\langle J(e_{l}), e_{6}\rangle^{2} \langle J(e_{\beta}), e_{i}\rangle^{2}\nonumber\\
&-\frac{3 \pi}{4}h'(0)\sum_{l=1}^{6}\sum_{i=1}^{5}\langle J(e_{i}), e_{6}\rangle^{2} \langle J(e_{l}), e_{6}\rangle^{2}
-\frac{10 \pi ^3}{3}h'(0)\sum_{l=1}^{6}\sum_{\nu,i=1}^{5}\langle J(e_{\nu}), e_{6}\rangle^{2} \langle J(e_{l}), e_{i}\rangle^{2}\nonumber
\end{align}
\begin{align}
&-\frac{8\pi ^3}{3}h'(0)\sum_{l=1}^{6}\sum_{\nu,i=1}^{5}\langle J(e_{l}), e_{6}\rangle^{2} \langle J(e_{\nu}), e_{i}\rangle^{2}
-3\pi ^3h'(0)\sum_{i=1}^{5}\langle J(e_{6}), e_{6}\rangle \langle J(e_{i}), e_{i}\rangle\nonumber\\
&-\frac{4\pi ^3}{3}h'(0)\sum_{\nu,i=1}^{5}\langle J(e_{i}), e_{i}\rangle\langle J(e_{\nu}), e_{\nu}\rangle
+\frac{3 \pi}{10}\left(128 \pi ^2+5\right)h'(0)\sum_{l=1}^{6}\langle J(e_{l}), e_{6}\rangle^{2}\nonumber\\
&-5\pi ^3h'(0)\sum_{i=1}^{5}\langle J(e_{i}), e_{6}\rangle^{2}
+\frac{28 \pi ^3}{15}h'(0)\sum_{l=1}^{6}\sum_{i=1}^{5}\langle J(e_{l}), e_{i}\rangle^{2}\nonumber\\
&-\frac{28 \pi ^3}{5}h'(0)\sum_{l,\beta=1}^{6}\sum_{i=1}^{5}\langle J(e_{l}), e_{i}\rangle\ \langle J(e_{\beta}), e_{i}\rangle \langle J(e_{l}), e_{6}\rangle \langle J(e_{\beta}), e_{6}\rangle \nonumber\\
&+4\pi^3h'(0)\sum_{l=1}^{6}\sum_{\nu,i=1}^{5} \langle J(e_{\nu}), e_{i}\rangle \langle J(e_{l}), e_{i}\rangle \langle J(e_{\nu}), e_{6}\rangle \langle J(e_{l}), e_{6}\rangle \nonumber\\
&+\frac{\pi }{4}h'(0)\sum_{l=1}^{6}\sum_{i=1}^{5} \langle J(e_{l}), e_{6}\rangle^{2} \langle J(e_{i}), e_{i}\rangle \langle J(e_{6}), e_{6}\rangle\nonumber\\
&+2\pi^3h'(0)\sum_{l=1}^{6}\sum_{\nu,i=1}^{5}\langle J(e_{l}), e_{i}\rangle^{2} \langle J(e_{\nu}), e_{\nu}\rangle \langle J(e_{6}), e_{6}\rangle
\Big]\Omega_4d{\rm Vol_{\partial M}},\nonumber
\end{align}
where $s$ is the scalar curvature.
\end{thm}

\section{Acknowledgements}

The author was supported in part by  NSFC No.11771070. The author thanks the referee for his (or her) careful reading and helpful comments.

\vskip 1 true cm


\bigskip
\bigskip

\noindent {\footnotesize {\it S. Liu} \\
{School of Mathematics and Statistics, Northeast Normal University, Changchun 130024, China}\\
{Email: liusy719@nenu.edu.cn}

\noindent {\footnotesize {\it Y. Wang} \\
{School of Mathematics and Statistics, Northeast Normal University, Changchun 130024, China}\\
{Email: wangy581@nenu.edu.cn}

\clearpage
\section*{Statements and Declarations}

Funding: This research was funded by National Natural Science Foundation of China: No.11771070.\\

Competing Interests: The authors have no relevant financial or non-financial interests to disclose.\\

Author Contributions: All authors contributed to the study conception and design. Material preparation, data collection and analysis were performed by Siyao Liu and Yong Wang. The first draft of the manuscript was written by Siyao Liu and all authors commented on previous versions of the manuscript. All authors read and approved the final manuscript.\\

Availability of Data and Material: The datasets supporting the conclusions of this article are included within the article and its additional files.\\

\end{document}